\documentclass{amsart}
\usepackage{shortsalch, xy, eufrak}
\usepackage{tikz}
\xyoption{all}
\usepackage{txfonts,cleveref}
\title{$KU$-local zeta-functions of finite CW-complexes.}
\date{June 2023.}
\author{A. Salch}

\begin{document}
\begin{abstract}
Begin with the Hasse-Weil zeta-function of a smooth projective variety over $\mathbb{Q}$. 
Replace the variety with a finite CW-complex, replace \'{e}tale cohomology with complex $K$-theory $KU^*$, and replace the $p$-Frobenius operator with the $p$th Adams operation on $K$-theory. This simple idea yields a kind of ``$KU$-local zeta-function'' of a finite CW-complex. For a wide range of finite CW-complexes $X$ with torsion-free $K$-theory, we show that this zeta-function admits analytic continuation to a meromorphic function on the complex plane, with a nice functional equation, and whose special values in the left half-plane recover the $KU$-local stable homotopy groups of $X$ away from $2$.

We then consider a more general and sophisticated version of the $KU$-local zeta-function, one which is suited to finite CW-complexes $X$ with nontrivial torsion in their $K$-theory. This more sophisticated $KU$-local zeta-function involves a product of $L$-functions of complex representations of the torsion subgroup of $KU^0(X)$, similar to how the Dedekind zeta-function of a number field factors as a product of Artin $L$-functions of complex representations of the Galois group. For a wide range of such finite CW-complexes $X$, we prove analytic continuation, and we show that the special values in the left half-plane recover the $KU$-local stable homotopy groups of $X$ away from $2$ if and only if the skeletal filtration on the torsion subgroup of $KU^0(X)$ splits completely.
\end{abstract}
\maketitle

\section{Introduction.}

\subsection{The main ideas and results.}
\label{The main ideas and results.}

Recall (e.g. Theorem 8.10 from \cite{MR737778}) the calculation of the $KU[1/2]$-local stable homotopy groups of spheres\footnote{The reader who is not already familiar with stable homotopy groups of Bousfield localizations is advised to skip ahead to \cref{The broader program}, below, where we give a brief introduction to the idea, and its place within computational stable homotopy theory.} by Adams-Baird and Ravenel:
\begin{theorem}\label{ravenel calc}
The ring of homotopy groups $\pi_*(L_{KU[1/2]}S^0)$ of the $KU[1/2]$-local sphere is determined by the following:
\begin{itemize}
\item $\pi_0(L_{KU[1/2]}S^0) \cong \mathbb{Z}[1/2]$.
\item $\pi_{-1}(L_{KU[1/2]}S^0) \cong 0$.
\item $\pi_{-2}(L_{KU[1/2]}S^0) \cong (\mathbb{Q}/\mathbb{Z})[1/2]$.
\item For all $n>0$, $\pi_{2n-1}(L_{KU[1/2]}S^0)$ is a cyclic group of order equal to the denominator\footnote{We adopt the convention that the denominator of $0\in\mathbb{Q}$ is $1$.} of the rational number $\zeta(1-n)$, up to a power of $2$. Here $\zeta(1-n)$ is the value of the Riemann zeta-function at $1-n$.
\item For all $n>0$, $\pi_{2n}(L_{KU[1/2]}S^0)$ is trivial.
\item For each integer $n$, the multiplication map $\pi_n(L_{KU[1/2]}S^0) \times \pi_{-2-n}(L_{KU[1/2]}S^0) \rightarrow \pi_{-2}(L_{KU[1/2]}S^0) \cong (\mathbb{Q}/\mathbb{Z})[1/2]$ is a perfect pairing. In particular, for positive $n$, $\pi_{-1-2n}(L_{KU[1/2]}S^0)$ is also cyclic of order equal to the denominator of $\zeta(1-n)$, up to a power of $2$.
\item For each $n$, the multiplication map $\pi_n(L_{KU[1/2]}S^0) \times \pi_{j-n}(L_{KU[1/2]}S^0) \rightarrow \pi_{j}(L_{KU[1/2]}S^0)$ is zero unless $j=-2$ or $j=n$ or $n=0$.
\end{itemize}
\end{theorem}

It is appealing to be able to describe the orders of Bousfield-localized stable homotopy groups of a finite CW-complex in terms of special values of a zeta-function of some kind. Very few results of this kind are already known\footnote{On the other hand, Bousfield-localized stable homotopy groups of {\em algebraic $K$-theory spectra} are well-known to admit deep relationships to special values of zeta-functions: see Example 4.8 of \cite{MR826102}, for example. But even after a Bousfield localization, algebraic $K$-theory spectra are almost never finite CW-complexes (except in the case of the algebraic $K$-theory of a finite field). We are compelled, by the classical topological applications of stable homotopy groups, to try to understand the stable homotopy groups of finite CW-complexes, and most importantly, spheres: for example, the stable homotopy groups of spheres are the attaching maps for stable $2$-cell complexes, so to have any hope of solving the fundamental topological problem of classifying the homotopy types of finite CW-complexes, one must determine the stable homotopy groups of spheres! To classify the stable homotopy types of the $3$-cell complexes whose $2$-skeleton is a fixed $2$-cell complex $X$, this task amounts to calculating the stable homotopy groups of $X$; and so on. The point is that calculating stable homotopy groups of finite CW-complexes is of fundamental importance. See \cref{The broader program} for further discussion of topological applications of stable homotopy groups of Bousfield-localized finite CW-complexes.}: the only other example is \cite{rmjpaper}, where for odd primes $p$, it was shown that the orders of the $KU$-local stable homotopy groups of the mod $p$ Moore spectrum $S^0/p$ are given by denominators of special values of $\zeta_F(s)/\zeta(s)$, where $\zeta(s)$ is the Riemann zeta-function, and $\zeta_F(s)$ is the Dedekind zeta-function of the smallest subfield $F$ of the cyclotomic field $\mathbb{Q}(\zeta_{p^2})$ in which $p$ is wildly ramified.

In this paper we prove that the orders of the $KU$-local stable homotopy groups of a much wider class of finite CW-complexes are given by special values of certain zeta-functions. The most important definition is Definition \ref{def of ku-local l and zeta}, in which we define a {\em $KU$-local zeta-function}, written $\zeta_{KU}(s,X)$, for every finite CW-complex $X$ such that
\begin{itemize}
\item
 the complex $K$-theory $KU^*(X)$ is concentrated in even degrees, and 
\item
 the order of the torsion subgroup of $KU^0(X)$ is square-free. 
\end{itemize}
The zeta-function $\zeta_{KU}(s,X)$ is defined as a product of two factors, a ``provisional $KU$-local zeta-function'' $\dot{\zeta}_{KU}(s,X)$, and a ``torsion $KU$-local $L$-function'' $L_{\tors KU}(s,X)$:
\begin{align}
\nonumber \zeta_{KU}(s,X) &= \dot{\zeta}_{KU}(s,X)\cdot L_{\tors KU}(s,X),\mbox{\ \ \ where} \\
\label{zeta eq dot} \dot{\zeta}_{KU}(s,X) &= \prod_{p} \frac{1}{\det\left(\id - p^{-s}\Psi^p\mid_{\hat{KU}_{\ell_p}^n(X)}\right)},\mbox{\ \ \ and} \\
\label{l eq} L_{\tors KU}(s,X) &= \prod_{w\in \mathbb{Z}}\prod_{\rho_w}\prod_{p} \frac{1}{\det\left(\id - p^{w-s}\Psi_{\rho_w,p,i_w}(p)\right)}.
\end{align}
The Euler products \eqref{zeta eq dot} and \eqref{l eq} are understood to be valid only in a suitable right-hand half-plane. The former, \eqref{zeta eq dot}, is a straightforward Euler product of characteristic polynomials of Adams operations acting on $K$-theory of $X$ completed at a prime at which the $K$-theory is torsion-free. By contrast, the Euler product \eqref{l eq} is a product of characteristic polynomials of Adams operations acting on the torsion in the $K$-theory of $X$, taken not merely over prime numbers $p$, but also {\em prime representations} $\rho$ of the torsion subgroup $\tors KU^0(X)$ of the $K$-theory of $X$; in this paper, a complex representation is said to be {\em prime} if it is induced up from a subgroup of prime order. See \cref{section on even and odd} for details.

In Proposition \ref{definedness of tors zeta}, $\zeta_{KU}(s,X)$ is shown to analytically continue to a meromorphic function on the complex plane. Proposition \ref{definedness of tors zeta} also yields a factorization of $\zeta_{KU}(s,X)$ as a product of shifts (i.e., Tate twists) of $L$-functions of even Dirichlet characters. Its functional equation is discussed in \cref{Functional equations for tors}. In Example \ref{cellular variety example} we also see that, for a smooth projective cellular variety $V$ over $\mathbb{Q}$ whose complex analytic space $\mathbb{C}(V)$ has rational cohomology concentrated in even degrees, the provisional $KU$-local zeta-function $\dot{\zeta}_{KU}(s,\mathbb{C}(V))$ of the {\em space} $\mathbb{C}(V)$ recovers the Hasse-Weil zeta-function $\zeta_V(s)$ of the {\em variety} $V$. 

The main result in \cref{section on even and odd}, and in this paper, is a formula for $KU$-local stable homotopy groups of $X$ in terms of special values of $\zeta_{KU}(s,X)$ at negative integers. To state the result, we must explain the factorization of $\zeta_{KU}(s,X)$ into {\em isoweight factors.} The factor $\dot{\zeta}_{KU}(s,X)$ of $\zeta_{KU}(s,X)$, which is sensitive to the {\em rational} $K$-theory of $X$, splits as a product $\prod_{w\in \mathbb{Z}} \dot{\zeta}^{(w)}_{KU}(s,X)$ of weight $w$ factors, one for each integer $w$. Similarly, the factor $L_{\tors KU}(s,X)$ of $\zeta_{KU}(s,X)$, which is sensitive to the {\em torsion} in the $K$-theory of $X$, splits as a product $\prod_{w\in \mathbb{Z}} \prod_{\ell\mid n_w} L_{\tors KU}^{(w,\ell)}(s,X)$ of weight $(w,\ell)$ factors, where $\ell$ is an odd prime divisor of $n_w$, the order of the $2w$th filtration quotient in the skeletal filtration of the torsion subgroup of $KU^0(X)$. We now state Theorem \ref{main spec vals thm 2} and Corollary \ref{main spec vals thm 2 cor}, which identify orders of $KU$-local stable homotopy groups in terms of denominators of special values of these ``isoweight'' factors:
\begin{unnumberedtheorem}
Let $X$ be a finite CW-complex whose complex $K$-theory $KU^*(X)$ is concentrated in even degrees, and such that the torsion subgroup of $KU^0(X)$ has square-free order.
Let $a,b$ be the least and greatest integers $n$, respectively, such that $H^{2n}(X;\mathbb{Z})$ is nontrivial.
Then the following conditions are equivalent:
\begin{enumerate}
\item The filtration of $\tors KU^0(X)$ by the skeleta of the CW-complex $X$ splits completely, i.e., each of the projections $(\tors KU^0(X^n))^{\wedge}_p\rightarrow (\tors KU^0(X^{n-1}))^{\wedge}_p$ is a split surjection of $\hat{\mathbb{Z}}_p[\hat{\mathbb{Z}}_p^{\times}]$-modules for each odd prime $p$. Here $\hat{\mathbb{Z}}_p^{\times}$ acts on $p$-complete $K$-theory via Adams operations.
\item 
For all odd integers $2k-1$ satisfying $2k-1> 1-2a$, the $KU$-local stable homotopy group $\pi_{2k-1}(L_{KU}DX)$ of the Spanier-Whitehead dual $DX$ of $X$ is finite, and up to powers of $2$, its order is equal to the product
\begin{equation*}
 \prod_{w\in\mathbb{Z}}\left( \denom\left(\dot{\zeta}^{(w)}_{KU}(1-k,X)\right)\cdot \prod_{\ell\mid n_w}\denom\left( L^{(w,\ell)}_{\tors KU}(1-k,X)\right)\right),\end{equation*}
up to powers of $2$.
\item 
For all odd integers $2k-1$ satisfying $2k-1< -2b-3$, the $KU$-local stable homotopy group $\pi_{2k-1}(L_{KU}DX)$ is finite of order 
\begin{equation*}
\prod_{w\in\mathbb{Z}}\left( \denom\left(\dot{\zeta}^{(w)}_{KU}(k+1,X)\right)\cdot \prod_{\ell\mid n_w}\denom\left( L^{(w,\ell)}_{\tors KU}(k+1,X)\right)\right)\end{equation*}
up to powers of $2$.
\end{enumerate}
\end{unnumberedtheorem}
\begin{unnumberedcorollary}
Suppose that the the skeletal filtration of $\tors KU^0(X)$ splits completely. Write $N$ for the order of the group $\tors KU^0(X)$.
Then, for all odd integers $2k-1\geq 1-2a$, the order of $\pi_{2k-1}(L_{KU}DX)$ is equal to the denominator of $\zeta_{KU}(1-k,X)$ up to powers of $2$ and powers of $F$-irregular primes, where $F$ ranges across the wildly ramified subfields of the cyclotomic field $\mathbb{Q}\left(\zeta_{N^2}\right)$.
\end{unnumberedcorollary}
``$F$-irregular primes'' are studied in algebraic number theory, e.g. in \cite{MR0332730} and \cite{MR1138169}. The $\mathbb{Q}$-irregular primes are merely the ordinary irregular primes. A brief definition of the $F$-irregular primes is given preceding Corollary \ref{main spec vals thm 2 cor}.

Theorem \ref{main spec vals thm 2} and Corollary \ref{main spec vals thm 2 cor} demonstrate that the Adams-Baird-Ravenel calculation, Theorem \ref{ravenel calc}, is not an isolated phenomenon, but in fact, $KU$-local stable homotopy groups of finite CW-complexes are quite commonly expressible in terms of special values of zeta-functions which are
\begin{itemize}
\item constructed a natural way (mimicking Hasse-Weil zeta-functions---see below),
\item have good properties (e.g. analytic continuation due to Proposition \ref{definedness of tors zeta}, functional equations due to Theorem \ref{func eq 1a}  and the discussion in \cref{Functional equations for tors}),
\item and are directly connected to number theory (since they are nontrivially equal to products of shifts of $L$-functions of Dirichlet characters, by Proposition \ref{definedness of tors zeta}) and arithmetic geometry (since $\zeta_V(s) = \dot{\zeta}_{KU}(s,\mathbb{C}(V))$ for many cellular varieties $V$).
\end{itemize}

We remark that the definitions and results of \cref{section on even and odd} are more general than what is stated here in the introduction. In \cref{section on even and odd}, we allow for a set of primes $P$ at which $KU^*(X)$ is {\em not} necessarily concentrated in even degrees, and has torsion subgroup which does not necessarily have square-free order. We define an ``away-from-$P$'' $KU$-local zeta-function $\zeta_{KU[P^{-1}]}(s,X)$ such that the denominators of the special values of $\zeta_{KU[P^{-1}]}(s,X)$ at negative integers are related to the orders of the $KU$-local stable homotopy groups of the Spanier-Whitehead dual of $X$, {\em up to powers of $2$ and primes in $P$}. Consequently, the methods and results of \cref{section on even and odd} can be applied to {\em all} finite CW-complexes.

In this paper, the exposition is oriented toward building up the ``$KU$-local zeta functions'' of finite CW-complexes in an incremental way: in \cref{The provisional...}, we begin with the Hasse-Weil zeta-function of a smooth projective variety over $\mathbb{Q}$, and we make the simple change of replacing the Weil cohomology with complex $K$-theory, and replacing the $p$-Frobenius operator on the Weil cohomology with the $p$th Adams operator on $K$-theory. This yields the ``provisional $KU$-local zeta-function'' $\dot{\zeta}_{KU}(s,X)$ which, as mentioned above, is sensitive only to information visible to the rational $K$-theory of $X$. From there, we make amendments and improvements to the provisional $KU$-local zeta-function, eventually arriving at the (non-``provisional'') $KU$-local zeta-function $\zeta_{KU}(s,X)$ in \cref{section on even and odd}. The author hopes that this step-by-step method of exposition makes the definitions seem more natural, makes the motivations more obvious, and makes it more satisfying when we find (in Theorems \ref{main spec vals thm} and \ref{main spec vals thm 2}) that the special values of these zeta-functions count orders of $KU$-local stable homotopy groups.

The subject matter of this paper has significant overlap with that of \cite{MR4384614}, but the constructions and results in this paper are quite different from those of \cite{MR4384614}, and the aim of \cite{MR4384614} is the opposite of what this paper sets out to do. In \cite{MR4384614}, Zhang begins with a ($p$-adic) Dirichlet character $\chi$, and from that character, constructs a $KU$-local spectrum whose homotopy groups are describable in terms of the denominators of the $L$-values of $\chi$ at negative integers. However, most $KU$-localizations of finite CW-complexes---e.g. the $KU$-localization of any spectrum whose rational homotopy groups have rank $>1$---do not arise from a character in this way. By contrast, in the present paper, we begin with a finite CW-complex, and from it, we build a zeta-function (which, in the end, will always be equal to a product of shifts of Dirichlet $L$-functions, but this is nontrivial!) such that the denominators of its $L$-values recover the orders of the $KU$-local homotopy groups of $X$.  So the ``dictionary'' we construct goes in the reverse direction from that of \cite{MR4384614}. Having ``dictionaries'' in both directions is worthwhile, and we hope the reader agrees that this paper and \cite{MR4384614} complement each other.

\subsection{The broader program.}
\label{The broader program}

Given a generalized homology theory $E_*$ and a spectrum $X$, by \cite{MR551009} there exists a {\em Bousfield localization of $X$ at $E_*$}, written $L_EX$. The spectrum $L_EX$ is defined by a certain universal property; see section 1 of \cite{MR737778} for a very approachable introduction to $L_E$ and its basic properties. In the special case $E_*$ is $\pi_*(-)[P^{-1}]$, i.e., stable homotopy groups with some set $P$ of prime numbers inverted, the effect of the Bousfield localization $L_EX$ on stable homotopy groups is merely to invert the primes in $P$. That is,
 $\pi_*\left(L_{\pi_*(-)[P^{-1}]}X\right) \cong \pi_*(X)[P^{-1}].$ 
In that sense, Bousfield localization generalizes the classical theory of localization in algebra, by which one inverts a set of primes. 

However, there are many more generalized homology theories $E_*$ than those of the form $\pi_*(-)[P^{-1}]$ for a set of primes $P$. For many choices of $E_*$, the relationship between $\pi_*X$ and $\pi_*L_EX$ is more mysterious and subtle than any classical algebraic localization, but yields data of great topological importance.

Here is the most pressing class of examples. For each prime $p$ and each nonnegative integer $n$, there exists a generalized homology theory $E(n)_*$, the {\em $p$-primary height $n$ Johnson-Wilson theory}. In the base case $n=0$, $E(0)_*$ is merely classical rational homology, $H_*(-;\mathbb{Q})$, regardless of the prime $p$. On the other hand, for all positive integers $n$, the generalized homology theory $E(n)_*$ depends on the choice of prime number $p$, but the choice of $p$ is traditionally suppressed from the notation $E(n)_*$.

In the case $n=1$, the generalized homology theory $E(1)_*$ is the Adams summand of $p$-local complex $K$-theory, and consequently we have isomorphisms\footnote{To be clear: the expression $\pi_*(L_{KU}X)_{(p)}$ in \eqref{isos 1} means the localization, in the classical sense, of the graded abelian group $\pi_*(L_{KU}X)$ at the prime $p$.}
\begin{equation}\label{isos 1} \pi_*(L_{E(1)}X) \simeq \pi_*(L_{KU_{(p)}}X) \simeq \pi_*(L_{KU}X)_{(p)}.\end{equation}
The $KU$-local and $E(1)$-local stable homotopy groups have been studied systematically in \cite{MR0796907} and \cite{MR1075335}. They have been useful: Thomason \cite{MR826102} showed that $KU$-local stable homotopy groups of certain algebraic $K$-theory spectra agree with the \'{e}tale $K$-theory groups of \cite{MR0648533} and \cite{MR0805962}, and used this fact to great effect in calculations.

The $E(n)$-localizations $L_{E(n)}X$ for $n>1$ are less familiar, but also important.
Fix a prime $p$ and a finite CW-complex $X$. The {\em chromatic tower} is a tower of spectra
\[ \dots \rightarrow L_{E(2)}X\rightarrow L_{E(1)} X\rightarrow L_{E(0)}X\]
whose homotopy limit is weakly equivalent \cite[Theorem 7.5.7]{MR1192553} to the $p$-localization of $X$. The calculation of the stable homotopy groups of each individual stage $L_{E(n)}X$ is, at least in principle, approachable by a sequence of spectral sequences and homotopy fiber squares (``fracture squares'') which begins with the continuous cohomology of the profinite automorphism group of a formal group law of height $h$ over a finite field \cite{MR2030586}, for each $h\leq n$.

In the base case, $X=S^0$, the chromatic tower plays a central role in many algebraic topologists' understanding of the stable homotopy groups of spheres, i.e., the ``stable stems'': the $p$-local stable stems are recoverable from the infinite sequence $\dots\rightarrow \pi_*L_{E(2)}S^0 \rightarrow \pi_*L_{E(1)}S^0 \rightarrow \pi_*L_{E(0)}S^0$, and each stage in the sequence is, at least in principle, calculable starting from certain cohomology calculations arising from formal group laws.  

For $n\geq 2$, the stable homotopy groups of $L_{E(n)}X$ have generally been so complicated that the outcome of making a long and difficult calculation of $\pi_*L_{E(n)}X$ is a theorem whose {\em statement}---let alone the proof!---is prohibitively long and complicated; see for example \cite{MR2914955} and the discussion there of the pioneering calculations of \cite{MR1318877}. It puts the subject in a difficult position when the outcomes of deep and important fundamental calculations are theorems which are extremely cumbersome to state, even to an audience of experts.

The description of $\pi_*(L_{E(1)}S^0)$ in terms of special values of $\zeta(s)$ (i.e., the $p$-local version of Theorem \ref{ravenel calc}) suggests a way forward: {\em rather than attempt to describe $\pi_*(L_{E(n)}X)$ in elementary terms for each $m$, we might write down a description of the order of $\pi_m(L_{E(n)}X)$, or of various natural summands of $\pi_m(L_{E(n)}X)$, in terms of special values of various $L$-functions or zeta-functions}. While $\pi_m(L_{E(n)}X)$ would remain a somewhat mysterious object, at least the mystery would be fruitfully identified with some other compelling and well-studied mystery. 
Theorem \ref{ravenel calc} demonstrates that even in the case $n=1$ and $X=S^0$, this approach already yields rewards, in the form of shorter and more natural {\em statements} of theorems.

The main theorem of this paper, Theorem \ref{main spec vals thm 2}, demonstrates that in the case $n=1$, there is a simple and natural special-values description of $\pi_*(L_{KU[1/2]}X)$ for a wide range of finite CW-complexes $X$, not merely the case $X=S^0$ demonstrated by Adams-Baird and Ravenel, and not merely the case $X = S^0/p$ demonstrated in \cite{rmjpaper}. The author regards this as incremental progress toward more fully realizing the perspective described in the previous paragraph. 

\subsection{The intended audience and scope of this paper.}

The author has made an effort to write this paper so that it can be read by algebraic topologists who do not already know a lot about zeta-functions, and also number theorists who do not already know a lot about stable homotopy theory. As a consequence, sometimes a notion is explained which is very elementary and well-known to one audience, but not to the other. The author apologizes to any reader who finds this annoying.

Beginning in \cref{section on even and odd}, the author assumes that the reader is comfortable with basic notions about Dirichlet characters, e.g. primitivity and conductors. Such material is covered in many textbooks, like \cite{MR0434929}, but section 3 of the recent paper \cite{rmjpaper} is intended to be a ``crash course'' in those ideas specifically suited for an audience of algebraic topologists, so perhaps that reference can be particularly useful to some readers.

\begin{remark}\label{remarkongenerality}
The results in this paper admit generalizations in several directions. One direction of generalization is to higher heights, i.e., describing $E$-local zeta-functions of finite CW-complexes, where $E$ is the spectrum representing a complex-orientable cohomology theory whose $p$-height is $>1$ for some primes $p$.  
However, that direction of generalization is entirely outside the scope of this paper: it is important to get the $KU$-local story right first, and that is what this paper tries to do.

Even within the $KU$-local story, in some places it is possible to generalize the results beyond what is described in this paper. For instance, beginning in \cref{Defining the... torsion...}, we restrict our attention to CW-complexes whose $K$-theory is concentrated in even degrees. This is not strictly necessary, but yields cleaner statements and proofs, and makes the ideas and their motivations more obvious. 
In this paper the author has chosen to prioritize simplicity and well-motivatedness of the constructions, rather than reaching for the very greatest generality.
\end{remark}

\begin{remark}
An alternative approach to producing $KU$-local $L$-functions and zeta-function of finite CW-complexes is to construct {\em $p$-adic} $L$-functions. Here is a sketch of the construction. One begins with a finite CW-complex $X$, splits the group $\hat{\mathbb{Z}}_p^{\times}$ of $p$-adic Adams operations as $\mathbb{F}_p^{\times}\times \hat{\mathbb{Z}}_p$, then considers the $\omega^i$-eigenspace $e_i\hat{KU}^0_p(X)$ of the action of a generator for $\mathbb{F}_p^{\times}$ on $\hat{KU}^0_p(X)$, where $\omega$ is a fixed primitive $(p-1)$th root of unity in $\hat{\mathbb{Z}}_p$. Using the structure theory for Iwasawa modules (see Theorem 13.12 of \cite{MR1421575} for a textbook account) and the finiteness of $X$, each $e_i\hat{KU}^0_p(X)$ is pseudo-isomorphic to a direct sum $\oplus_{j=1}^{d_i}\Lambda/f_{i,j}(T)$ of quotients of the Iwasawa algebra $\Lambda \cong \hat{\mathbb{Z}}_p[[T]]$ by characteristic polynomials $f_{i,1}(T), \dots ,f_{i,d_i}(T)$. If $d_i=1$ for all $i$, then we define $L_p(s,\omega^i) := f_i(q^{1-s}-1)$, with $q$ a principal unit in $\hat{\mathbb{Z}}_p^{\times}$ (see discussion preceding Theorem 11.6.7 of \cite{MR2392026} for a precise account), obtaining a set of $(p-1)$ $p$-adic $L$-functions of $X$; compare with Iwasawa's theorem $e_i(U/C)\sim \Lambda/f_i(T)$ as in Theorem 11.6.18 of \cite{MR2392026}. The resulting ``$KU$-local $p$-adic $L$-functions of finite CW-complexes'' have some good properties but are not guaranteed to be the $p$-adic interpolations of complex-analytic functions. The author and his student A. Maison are investigating the resulting theory of $p$-adic $L$-functions for finite CW-complexes, but we regard that theory as beyond the scope of this paper, which is concerned with the complex-analytic case instead.
\end{remark}

\begin{convention}
Throughout this paper, all our topological constructions on CW-complexes are {\em stable} constructions. Consequently, for the sake of the definitions and theorems in this paper, the reader is welcome to consider finite CW-{\em spectra} (i.e., arbitrary desuspensions of finite CW-complexes) as examples. For example, spheres of negative dimension are perfectly acceptable finite CW-complexes for the purposes of all the definitions and theorems in this paper! We also consistently write $\pi_*$ for {\em stable} homotopy groups. There are no unstable homotopy groups appearing in this paper.
\end{convention}

\section{The provisional $KU$-local zeta-function of a finite CW-complex.}
\label{The provisional...}

\subsection{Defining the provisional $KU$-local zeta-function.}
We begin with a cursory account of what global Hasse-Weil zeta-functions (henceforth simply called ``Hasse-Weil zeta-functions'') are. For a fuller story, Weil's original article \cite{MR0029393} is a very good starting place. 

When constructed via cohomology, the Hasse-Weil zeta-function of a smooth projective variety $X$ over $\mathbb{Q}$ begins as an Euler product over primes $p$ of good reduction for $X$:
\begin{equation}\label{hasse-weil def 1}\prod_{p\notin P} \prod_{n\geq 0} \det(\id - p^{-s}\Fr_p\mid_{H^n_{\et}(X/\overline{\mathbb{F}}_p;\mathbb{Q}_{\ell})})^{(-1)^{n+1}},\end{equation}
where $P$ is the set of primes of bad reduction, and 
where $H^n_{\et}(X/\overline{\mathbb{F}}_p;\mathbb{Q}_{\ell})$ is $\ell$-adic \'{e}tale cohomology theory, for a prime $\ell\neq p$, applied to $\overline{\mathbb{F}}_p$-points of an integral model for $X$. 
The notation $\Fr_p$ denotes the $p$th-power relative Frobenius operator acting on the \'{e}tale cohomology.
The Euler product \eqref{hasse-weil def 1} converges for complex numbers $s$ with $\mathfrak{Re}(s)>>0$, and is then (when all goes well, depending on $X$!) analytically continued to a meromorphic function on the complex plane. With a bit more trouble, one can also include appropriate $p$-local Euler factors at the primes $p\in P$ of bad reduction. Since the set $P$ is finite, such changes only affect finitely many Euler factors, hence do not affect many of the important analytic properties of the zeta-function. Special values of the resulting function $\zeta_X(s)$ are of deep interest in arithmetic geometry; for example, the strong form of the Birch-Swinnerton-Dyer conjecture predicts the behavior of $\zeta_E(s)$ at $s=1$, for $E$ an elliptic curve. 

The element $\Fr_p$ is a topological generator for the profinite Galois group $\Gal(\overline{\mathbb{F}}_p/\mathbb{F}_p)$. The definition \eqref{hasse-weil def 1} makes good sense because the $\ell$-adic cohomology of the variety $X$ has, at each prime $p\neq \ell$, a natural action of the topologically cyclic profinite group $\Gal(\overline{\mathbb{F}}_p/\mathbb{F}_p)$.

Described in that way, \'{e}tale cohomology resembles complex topological $K$-theory, $KU^*$. Recall that, for each prime number $\ell$, the $\ell$-adically complete complex topological $K$-theory $\hat{KU}_{\ell}$ admits an natural action of the profinite group $\hat{\mathbb{Z}}_{\ell}^{\times}$ of units in the $\ell$-adic integers. This is the action by {\em Adams operations,} introduced in \cite{MR0139178}. 

A comparison between Frobenius actions on $\ell$-adic \'{e}tale cohomology and Adams operations on $\ell$-adic $K$-theory is, of course, not a new idea at all: see Sullivan's paper \cite{MR0442930}, for example, or work of Quillen in \cite{MR0315016} establishing that, for $p\neq \ell$, the action of $\Psi^p$ on the $\ell$-completed complex $K$-theory spectrum $\hat{KU}_{\ell}$ agrees with the action of the $p$-Frobenius operator on the $\ell$-completed algebraic $K$-theory spectrum $\mathcal{K}(\overline{\mathbb{F}}_p)^{\hat{}}_{\ell}$ under the Suslin rigidity equivalence $\mathcal{K}(\overline{\mathbb{F}}_p)^{\hat{}}_{\ell}\simeq \hat{KU}_{\ell}$. In that sense, $\Psi^p$ really {\em is} the $p$-Frobenius operator. 

Our aim in this section is to take that idea seriously, by mimicking the construction of Hasse-Weil zeta-functions of varieties but using Adams operations in place of Frobenius operations, to yield some kind of Hasse-Weil-like zeta-functions of a {\em topological space} $X$ (not a variety!), and to derive useful topological invariants of $X$ from the special values of that zeta-function.
The idea is simple: write down the Euler product \eqref{hasse-weil def 1} for the Hasse-Weil zeta-function, but
\begin{itemize}
\item rather than $X$ being a smooth projective variety, let $X$ be a finite CW-complex,
\item for each prime $\ell$, replace $\ell$-adic \'{e}tale cohomology with $\ell$-adic complex topological $K$-theory,
\item for each prime $p$, replace the Frobenius operator $\Fr_p$ on $\ell$-adic \'{e}tale cohomology with the $p$th Adams operation $\Psi^p$ on $\ell$-adic complex topological $K$-theory,
\item and, since complex topological $K$-theory is $2$-periodic, rather than taking the Euler product over all degrees, we will merely take the Euler product over degrees $0$ and $1$.
\end{itemize}
Here is the resulting definition:
\begin{definition}\label{def of euler product}
Let $X$ be a finite CW-complex.
Let $P$ be the set of primes consisting of $2$ and all of the primes at which the cohomology of $X$ has $p$-torsion, i.e., 
\[ P = \{ 2\} \cup \left\{ p \mbox{\ prime}: H^*(X;\mathbb{Z})\mbox{\ has\ nontrivial\ $p$-torsion}\right\} .\]
For each prime number $p$, suppose we have chosen a prime number $\ell_p\notin P$ such that $p$ is a topological generator for the group of $\ell_p$-adic units $\hat{\mathbb{Z}}_{\ell_p}^{\times}$. 
Then the {\em $KU$-theoretic Euler product for $X$} is defined as the product over all primes $p$:
\begin{align}\label{euler product 1} 
\prod_{p} \prod_{n\in \{0,1\}} \det(\id - p^{-s}\Psi^p\mid_{\hat{KU}_{\ell_p}^n(X)})^{(-1)^{n+1}}
 &=\prod_{p} \frac{\det\left( \id - p^{-s}\Psi^p\mid_{\hat{KU}_{\ell_p}^1(X)}\right)}{\det\left( \id - p^{-s}\Psi^p\mid_{\hat{KU}_{\ell_p}^0(X)}\right)}.
\end{align}
\end{definition}
The prime $2$ will have to be avoided or excluded in various small ways later in this paper, simply due to the slightly different properties of $2$-local $K$-theory as compared to $K$-theory localized at an odd prime (compare Theorem 8.10 to Theorem 8.15 in \cite{MR737778}, for example). This happened in a small way in Definition \ref{def of euler product} by including $2$ in the set of primes $P$. Note that this did not cause any missing Euler factors in \eqref{euler product 1}: the $2$-local Euler factor is still present in \eqref{euler product 1}. The author suspects that the prime $2$ can be incorporated elegantly into the theory presented in this paper, by using methods along the lines of Bousfield's ``united $K$-theory'' \cite{MR1075335}. That extension of the theory goes beyond the scope of this paper, though.

In \eqref{euler product 1}, we chose to use $\ell_p$-adically completed $K$-theory to maintain the similarity with the Euler product \eqref{hasse-weil def 1} of a Hasse-Weil zeta-function. It would have worked just as well to use $KU[P^{-1},p^{-1}]^*$, that is, complex $K$-theory with:
\begin{itemize}
\item $p$ inverted, so that the stable Adams operation $\Psi^p$ is defined, and 
\item all the primes in $P$ inverted, so that $KU^*(X)[P^{-1},p^{-1}]$ is a torsion-free (hence free) $\mathbb{Z}[P^{-1},p^{-1}]$-module\footnote{If $KU^*(X)[P^{-1},p^{-1}]$ were not free, we would have to exercise a bit of care about what the determinant of $\Psi^p$ acting on $KU^n(X)[P^{-1},p^{-1}]$ ought to mean. We generalize and extend these ideas to handle torsion in $K$-theory starting in \cref{section on even and odd}.}.
\end{itemize}
This change would not affect the determinants in \eqref{hasse-weil def 1}, hence would not affect the resulting zeta-function or the theorems we prove about it, below.

\begin{remark}
The author would like to emphasize how na\"{i}ve the Euler product \eqref{euler product 1} is: it comes from blindly mimicking the Euler product of the Hasse-Weil zeta-function, and in the process, losing hold of any clear interpretation in terms of Lefschetz fixed-point theory. 
\end{remark}

\begin{lemma}\label{euler product of a wedge sum}
The $KU$-theoretic Euler product of a wedge sum $X\vee Y$ is equal to the product of the $KU$-theoretic Euler products of $X$ and of $Y$.
\end{lemma}
\begin{proof}
Elementary.
\end{proof}

Theorem \ref{thm on existence of zeta functions} establishes the basic properties of the $KU$-theoretic Euler product, which are all straightforward consequences of the basic properties of the Chern character:
\begin{theorem}\label{thm on existence of zeta functions}
Let $X$ be a finite CW-complex. Let $P$ be as in Definition \ref{def of euler product}.
Write $b$ for the greatest integer $n$ such that at least one of the two vector spaces $H^{2n}(X;\mathbb{Q}), H^{2n+1}(X;\mathbb{Q})$ is nontrivial.
Then the following claims are each true:
\begin{itemize} 
\item The $KU$-theoretic Euler product \eqref{euler product 1} converges absolutely for all complex numbers $s$ with $\mathfrak{Re}(s) > 1+b$.
\item The $KU$-theoretic Euler product \eqref{euler product 1} analytically continues to a meromorphic function $\dot{\zeta}_{KU}(s,X)$ on the complex plane. 
\item The meromorphic function $\dot{\zeta}_{KU}(s,X)$ is equal to the $L$-function of a cellular motive. That is, $\dot{\zeta}_{KU}(s,X)$ is equal to a product of ``shifts'' of the Riemann zeta-function $\zeta(s)$. The product is as follows:
\begin{align}\label{factorization 1}
  \dot{\zeta}_{KU}(s,X) &= \prod_{w\in\mathbb{Z}}\zeta(s-w)^{\beta_{2w}(X) - \beta_{2w+1}(X)},
\end{align}
where $\beta_n(X) = \dim_{\mathbb{Q}}H^{n}(X;\mathbb{Q})$ is the $n$th Betti number of $X$.
\item Some of the poles and some of the zeroes of $\dot{\zeta}_{KU}(s,X)$ in the complex plane occur at integers. The order of vanishing of $\dot{\zeta}_{KU}(s,X)$ at an integer $s=m$ is equal to 
\begin{equation}\label{order of poles 1} -\beta_{2m-2}(X) + \beta_{2m-1}(X) + \sum_{k\geq 1} \left( \beta_{2m+4k}(X) - \beta_{2m+4k+1}(X)\right).\end{equation}
If the rational cohomology of $X$ is concentrated in even degrees, then {\em all} poles of $\dot{\zeta}_{KU}(s,X)$ occur at integers, and their orders are calculated by formula \eqref{order of poles 1}.
\end{itemize}
\end{theorem}
\begin{proof}

Of course the determinant of the linear operator $\id - p^{-s}\Psi^p\mid_{(\hat{KU}_{\ell_p})^n(X)}$ acting on the free $\hat{\mathbb{Z}}_{\ell_p}$-module $(\hat{KU}_{\ell_p})^n(X)$ is unchanged by first passing to the fraction field $\mathbb{Q}_{\ell_p}$ of $\hat{\mathbb{Z}}_{\ell_p}$. However, since all the attaching maps in a minimal $S[P^{-1}]$-cell decomposition for $X[P^{-1}]$ are rationally trivial, the action of the Adams operation $\Psi^p$ on the rationalized $\ell_p$-adic $K$-groups 
\begin{align*}
  \mathbb{Q}\otimes_{\mathbb{Z}}(\hat{KU}_{\ell_p})^n(X) 
   &\cong (\ell_p^{-1}\hat{KU}_{\ell_p})^n(X)
\end{align*}
agrees with the Adams operations on the rationalized $\ell_p$-adic $K$-groups of a wedge of spheres. We have Adams-operation-preserving isomorphisms
\begin{align*} 
 (\ell_p^{-1}\hat{KU}_{\ell_p})^0(X) &\cong \coprod_{n\in\mathbb{Z}} \ell_p^{-1}(\hat{KU}_{\ell_p})^0\left(S^{2n}(H^{2n}(X;\mathbb{Q}))\right)\mbox{\ \ and}\\
 (\ell_p^{-1}\hat{KU}_{\ell_p})^1(X) &\cong \coprod_{n\in\mathbb{Z}} \ell_p^{-1}(\hat{KU}_{\ell_p})^1\left(S^{2n+1}(H^{2n+1}(X;\mathbb{Q}))\right),\end{align*}
where $S^m(V)$ denotes the wedge of spheres whose rational homology in degree $m$ is $V$, and whose homology in all other degrees is trivial\footnote{This kind of observation is, of course, quite old: it is essentially just the statement of what the Chern character does to Adams operations, and it is why Bousfield imposes ``rational diagonalizability conditions'' on the objects of his categories $\mathcal{A}(p)$ and $\mathcal{B}(p)$ constructed in \cite{MR0796907}, for example.}. 

The point is that the $KU$-theoretic Euler product for $X$ agrees with the $KU$-theoretic Euler product for a wedge of spheres. The number of $n$-spheres in this wedge is equal to the $\mathbb{Q}$-linear dimension of $H^n(X;\mathbb{Q})$, i.e., $\beta_{n}$. By Lemma \ref{euler product of a wedge sum}, the $KU$-theoretic Euler product of $X$ is equal to the product of the $KU$-theoretic Euler products of each of the spheres in that wedge. The Adams operation $\Psi^p$ acts on $(\hat{KU}_{\ell_p})^0(S^{2n})\cong \hat{\mathbb{Z}}_{\ell_p}$ as multiplication by $p^n$. Hence the Euler factor for $S^{2n}$ at the prime $p$ is $\frac{1}{1 - p^{-s}p^n} = \frac{1}{1 - p^{n-s}}$, i.e., it is the Euler factor of $S^0$ with $s-n$ ``plugged in'' for $s$. Similarly, the Euler factor for $S^{2n+1}$ is $1-p^{n-s}$.
This yields that the Euler product for $X$ factors as 
$\prod_{w\in\mathbb{Z}}\prod_{p} \left( \frac{1}{1-p^{w-s}}\right)^{\beta_{2w}(X) - \beta_{2w+1}(X)}$,
which is precisely the Euler product for \eqref{factorization 1}. The third claim is now proven.

The product \eqref{factorization 1} converges absolutely for all complex numbers $s$ with $\mathfrak{Re}(s)>1+b$, since the Riemann zeta-function converges absolutely and is zero-free to the right of the line $\mathfrak{Re}(s)=1$. Each of the factors $\zeta(s-w)$ in \eqref{factorization 1} analytically continues to a meromorphic function on the complex plane, so the finite product \eqref{factorization 1} of those factors also analytically continues to the plane. By uniqueness of analytic continuation, the meromorphic function $\dot{\zeta}_{KU}(s,X)$ is well-defined. This proves the first and second claims.

The fourth claim, about the poles of $\dot{\zeta}_{KU}(s,X)$, follows from the fact that the only pole of the Riemann zeta-function in the complex plane occurs at $s=1$, and the fact that the only poles of $\frac{1}{\zeta(s)}$ at integers are the trivial zeroes of $\zeta(s)$, which are simple, and occur at all negative even integers. 
\end{proof}

Theorem \ref{thm on existence of zeta functions} suggests that we define the $KU$-local zeta-function of a finite CW-complex as the analytic continuation of the $KU$-theoretic Euler product \eqref{euler product 1}. This seems natural from the point of view of Hasse-Weil zeta-functions: after all, we were led to the $KU$-theoretic Euler product by a very simple analogy with the Hasse-Weil zeta-function. It is reasonable to study the analytic continuation described in Theorem \ref{thm on existence of zeta functions}, but when we reach \cref{section on even and odd}, we will have good reason to define the $KU$-local zeta-function of a finite CW-complex as a modification and refinement of that analytic continuation, one which pays attention to torsion in $K$-theory. For now, we will take the analytic continuation of \eqref{euler product 1} as a {\em provisional} version of the $KU$-local zeta-function:
\begin{definition}\label{def of provisional ku-local zeta}
Let $P$ be a finite set of primes, with $2\in P$.
Suppose that $X$ is a finite CW-complex.
We refer to the meromorphic function $\dot{\zeta}_{KU}(s,X)$ of Theorem \ref{thm on existence of zeta functions} as the {\em provisional $KU$-local zeta function of $X$.}
%
That is,
\begin{align}
\label{provisional zeta 1}  \dot{\zeta}_{KU}(s,X) &= \prod_{w\in\mathbb{Z}}\zeta(s-w)^{\beta_{2w}(X) - \beta_{2w+1}(X)}.
\end{align}

Finally, for an integer $w$, we write 
\begin{align*}
  \dot{\zeta}_{KU}^{(w)}(s,X) &= \zeta(s-w)^{\beta_{2w}(X) - \beta_{2w+1}(X)},
\end{align*}
and we refer to $ \dot{\zeta}_{KU}^{(w)}(s,X)$ as the {\em weight $w$ factor in $\dot{\zeta}_{KU}(s,X)$.} Any single such factor will be called an {\em isoweight} factor.
\end{definition}
The provisional $KU$-local zeta-function is clearly a very crude invariant of a finite CW-complex. If $X$ and $Y$ are finite CW-complexes whose rational cohomology is concentrated in even degrees (which will be the case of greatest interest for much of the rest of this paper), then $\dot{\zeta}_{KU}(s,X) = \dot{\zeta}_{KU}(s,Y)$ if and only if $X$ and $Y$ are {\em rationally} stably homotopy-equivalent. In other words, $\dot{\zeta}_{KU}(s,-)$ is really a $H\mathbb{Q}$-local invariant, not only a $KU$-local invariant. It is a peculiar fact that, on the finite CW-complexes whose cohomology is torsion-free and concentrated in even degrees, the orders of the $KU$-local stable homotopy groups are {\em also} a $H\mathbb{Q}$-local invariant, and in fact are recoverable from the special values of $\dot{\zeta}_{KU}(s,-)$: see Theorem \ref{thm on existence of zeta functions}, below.


\begin{example}\label{cellular variety example}
Here is a very simple example of a provisional $KU$-local zeta-function of a finite CW-complex. It is easy to use the methods in this section to see that, for any integer $n$, the complex projective space $\mathbb{C}P^n$ satisfies
\begin{align}
\label{eq f4a3} \dot{\zeta}_{KU}(s,\mathbb{C}P^n) &= \prod_{w=0}^n \zeta(s-w).
\end{align}
This is precisely the Hasse-Weil zeta-function of the projective space $P^n$ regarded as a variety over $\mathbb{Q}$. 

A trivial zero in one zeta-factor in \eqref{eq f4a3} may cancel with the pole in another zeta-factor, occasionally yielding amusing calculations like
\begin{align*}
 \dot{\zeta}_{KU}(1,\mathbb{C}P^3) 
  &= \zeta(1)\zeta(0)\zeta(-1)\zeta(-2) \\
  & = \frac{-\gamma\cdot \zeta(3)}{96\pi^2},
\end{align*}
where $\gamma$ is the Euler constant. This example demonstrates that, while $\zeta(s)$ has a pole at $s=1$, it is not always true that $\dot{\zeta}_{KU}(s,X)$ has a pole at $s=1$.


More generally, suppose $V$ is a smooth projective {\em cellular} variety over $\mathbb{Q}$ with associated complex analytic space $\mathbb{C}(V)$. Suppose that the cohomology $H^*(\mathbb{C}(V);\mathbb{Q})$ is concentrated in even degrees. Then the provisional $KU$-local zeta-function of the {\em space} $\mathbb{C}(V)$ recovers the Hasse-Weil zeta-function of the {\em variety} $V$:
\begin{align}
\label{eq 1203} \dot{\zeta}_{KU}\left(s,\mathbb{C}(V)\right) &= \zeta_V(s).
\end{align}
One cannot expect \eqref{eq 1203} to generalize to non-cellular varieties $V$, like elliptic curves, since the zeta-functions of such varieties are not products of shifts of $\zeta(s)$. Put another way, the Hasse-Weil zeta-function of a non-cellular variety captures genuinely {\em arithmetic} information about the variety, not merely topological information. Hence for non-cellular $V$ one cannot expect to recover $\zeta_V(s)$ from a topological invariant of $\mathbb{C}(V)$ like $\dot{\zeta}_{KU}\left(s,\mathbb{C}(V)\right)$.
\end{example}

\subsection{The functional equation of the provisional $KU$-local zeta-function.}

For each given weight $w$, the weight $w$ factor $\dot{\zeta}_{KU}^{(w)}(s,X)$ satisfies a functional equation relating $\dot{\zeta}_{KU}^{(w)}(s,X)$ to $\dot{\zeta}_{KU}^{(w)}(1+w-s,X)$. This functional equation is easily extracted from the functional equation from the Riemann zeta-function:
\begin{equation}\label{func eq 1}
 \dot{\zeta}^{(w)}_{KU}(s,X) 
  = \left( 2^{s-w}\pi^{s-w-1}\sin\left( \frac{\pi(s-w-1)}{2}\right)\Gamma(1+w-s)\right)^{\beta_{2w}(X) - \beta_{2w+1}(X)}\dot{\zeta}^{(w)}_{KU}(1+w-s,X).
\end{equation}
The functional equation \eqref{func eq 1} is precisely of the usual form of the functional equation of the $L$-function of a weight $w$ motive $M$, which relates $L(s,M)$ to $L(1+w-s,M)$. 

One ought to compare this situation to the situation of the classical Hasse-Weil zeta-function of a smooth projective variety over $\mathbb{Q}$. That Hasse-Weil zeta-function {\em does} admit a functional equation---it is not only that each of its isoweight factors has a functional equation, but the whole zeta-function itself also does. But the existence of that functional equation relies crucially on the (Weil) cohomology of a smooth projective variety satisfying Poincar\'{e} duality. 

As an example, consider the case of the projective line. An elementary and classical calculation of the Hasse-Weil zeta-function yields $\zeta_{\mathbb{P}^1}(s) = \zeta(s)\zeta(s-1)$. If we agree to write $\hat{\zeta}_{\mathbb{P}^1}(s)$ for the completed zeta-function $\hat{\zeta}(s)\hat{\zeta}(s-1)$, where $\hat{\zeta}(s)$ is the completed Riemann zeta-function satisfying $\hat{\zeta}(s) = \hat{\zeta}(1-s)$, then we have
\begin{align}
\nonumber \hat{\zeta}_{\mathbb{P}^1}(s) 
  &= \hat{\zeta}(s)\hat{\zeta}(s-1)\\
\nonumber  &= \hat{\zeta}(1-s)\hat{\zeta}(2-s) \\
\label{eq 3o4im}  &= \hat{\zeta}_{\mathbb{P}^1}(2-s),
\end{align}
yielding a functional equation for $\zeta_{\mathbb{P}^1}(s)$. However, the equation \eqref{eq 3o4im} exchanged the weight zero $\hat{\zeta}$-factor coming from the Weil $H^0(\mathbb{P}^1)$ with its Poincar\'{e} dual weight $1$ $\hat{\zeta}$-factor coming from the Weil $H^2(\mathbb{P}^1)$! 

In the general story of $KU$-local zeta-functions of finite CW-complexes, we have not restricted our attention to finite CW-complexes satisfying any form of Poincar\'{e} duality, so we cannot expect to get a simple functional equation relating $\dot{\zeta}_{KU}(s,X)$ to $\dot{\zeta}_{KU}(n-s,X)$ for any single value of $n$. Without assuming some kind of self-duality properties on $X$, it seems we must compromise:
\begin{itemize}
\item either settle for having only a functional equation for each isoweight factor of $\dot{\zeta}_{KU}(s,X)$, 
\item or settle for having a functional equation relating $\dot{\zeta}_{KU}(s,X)$ to $\dot{\zeta}_{KU}(s,X^*)$ for some kind of dual $X^*$ of $X$.
\end{itemize}
We already gave the outcome of the first compromise above, in \eqref{func eq 1}. 

The second compromise yields a better result. One formulation of the functional equation for motivic $L$-functions is as follows: if $M$ is a motive, with dual motive $M^{\vee}$, then a functional equation relates $L(s,M)$ to $L(1-s,M^{\vee})$. 
The functional equation for $\dot{\zeta}_{KU}(s,X)$ has an especially nice expression in similar terms. Let us write $\hat{\dot{\zeta}}_{KU}(s,X)$ for the {\em completed provisional $KU$-local zeta-function of $X$,} which we define na\"{i}vely by replacing each zeta-factor $\zeta(s-w)$ in \eqref{provisional zeta 1} with its completion $\hat{\zeta}(s-w)$ in the classical sense:
\begin{align*}
 \hat{\dot{\zeta}}_{KU}(s,X)
  &:= \prod_{w\in \mathbb{Z}} \hat{\zeta}(s-w)^{\beta_{2w}(X) - \beta_{2w+1}(X)}\\
  &= \prod_{w\in \mathbb{Z}} \left( \pi^{-(s-w)/2}\Gamma(\frac{s-w}{2})\zeta(s-w)\right)^{\beta_{2w}(X) - \beta_{2w+1}(X)}.
\end{align*}
\begin{theorem}\label{func eq 1a} 
If we write $D$ for the Spanier-Whitehead dual of a finite spectrum $X$ with torsion-free cohomology concentrated in even degrees, then we have the functional equation
\begin{align}\label{func eq 2a}
 \hat{\dot{\zeta}}_{KU}(s,X) 
  &= \hat{\dot{\zeta}}_{KU}(1-s,DX).
\end{align}
\end{theorem}
\begin{proof}
Elementary from \eqref{func eq 1}, Theorem \ref{thm on existence of zeta functions}, and the fact that $H^n(DX;\mathbb{Q})\cong H^{-n}(X;\mathbb{Q})$.
\end{proof}

\subsection{Special values of the provisional $KU$-local zeta-function.}
\label{Torsion-free...concentrated in even dimensions}

We now study the special values of the provisional $KU$-local zeta-function of a finite CW-complex $X$ at negative integers. Theorem \ref{main spec vals thm} generalizes Theorem \ref{ravenel calc}, the Adams-Baird-Ravenel calculation of the $KU[1/2]$-local stable homotopy groups of spheres in terms of special values of $\zeta(s)$.
\begin{theorem}\label{main spec vals thm}  
Let $X$ be a finite CW-complex, and let $P$ be the set of primes defined in Definition \ref{def of euler product}. Suppose that $H^*(X; \mathbb{Z}[P^{-1}])$ is concentrated in even dimensions. Write $DX$ for the Spanier-Whitehead dual of $X$. Write $a,b$ for the least and greatest integers $n$, respectively, such that $H^{2n}(X;\mathbb{Q})$ is nontrivial. Then we have the following consequences:
\begin{itemize}
\item
The $KU$-local stable homotopy group $\pi_{2k}(L_{KU}DX)[P^{-1}]$ is trivial if either $2k\geq 1-2a$ or $2k<-2b-2$. 
\item 
If $2k-1\geq 1-2a$, then the group $\pi_{2k-1}(L_{KU}DX)$ is finite of order equal to\footnote{The integer \eqref{spec val formula 1} is indeed well-defined, i.e., the special value $\dot{\zeta}_{KU}^{(w)}(1-k,X)$ of the weight $w$ factor of $\dot{\zeta}_{KU}(1-k,X)$ is rational for all $2k-1\geq 1-2a$. Similar observations yield well-definedness of \eqref{spec val formula 1a}, \eqref{spec val formula 2}, and \eqref{spec val formula 2a}.}
\begin{equation}\label{spec val formula 1}
\prod_{w\in\mathbb{Z}} \denom\left(\dot{\zeta}_{KU}^{(w)}(1-k,X)\right),\end{equation} 
up to factors of primes in $P$. 
We furthermore have an equality
\begin{align}\label{spec val formula 1a} 
 \left| \pi_{2k-1}(L_{KU}DX)\right|
  &= \denom\left( \dot{\zeta}_{KU}(1-k,X)\right)\end{align}
 up to powers of irregular primes\footnote{In number theory, regular primes (and their complement, the irregular primes) are classical and well-studied, but since some of the imagined audience for this paper includes topologists who may not be familiar with regular primes, here is a quick recap. A prime $p$ is {\em irregular} if and only if $p$ divides the class number of the cyclotomic field $\mathbb{Q}(\zeta_p)$. Kummer showed that $p$ is irregular if and only if $p$ divides the numerator of the Bernoulli number $B_{2n} = -2n\cdot \zeta(1-2n)$ for some $n\in \{ 1, \dots ,\frac{p-3}{2}\}$. Equivalently, by the Kummer congruences on Bernoulli numbers: $p$ is irregular if and only if $p$ divides the numerator of $B_{2n}$ for some $n\geq 1$.} and primes in $P$.
\item
If $2k-1<-2b-2$, the group $\pi_{2k-1}(L_{KU}DX)$ is finite of order equal to 
\begin{equation}\label{spec val formula 2}
\prod_{w\in\mathbb{Z}} \denom\left(\dot{\zeta}_{KU}^{(w)}(1+k,X)\right).
\end{equation}
up to factors of primes in $P$. 
We furthermore have an equality
\begin{align}\label{spec val formula 2a} 
 \left| \pi_{2k-1}(L_{KU}DX)\right|
  &= \denom\left( \dot{\zeta}_{KU}(1+k,X)\right)\end{align}
 up to powers of irregular primes and primes in $P$.
\end{itemize}
\end{theorem}
\begin{proof}
A simple spectral sequence argument suffices. Because $P$ includes all primes at which the cohomology of $X$ has nontrivial torsion, the Atiyah-Hirzebruch spectral sequence for the generalized cohomology theory represented by $L_{KU[1/2]}S^0$ takes the form
\begin{align*}
 E_2^{s,t} \cong H^s(X;\mathbb{Z}[P^{-1}])\otimes_{\mathbb{Z}[1/2]}\pi_{-t}(L_{KU[1/2]}S^0) & \Rightarrow \pi_{-s-t}F\left(X[P^{-1}], L_{KU[1/2]}S^0\right)\\
   & \cong \pi_{-s-t}\left( L_{KU[1/2]}DX\right)[P^{-1}] \\
 d^r: E_r^{s,t} &\rightarrow E_r^{s+r,t-r+1}.
\end{align*}
The spectral sequence converges strongly, since $X$ is finite. Plotted with the Serre conventions, the spectral sequence $E_2$-term is of the following form:

\begin{minipage}{1.0\textwidth}
\centering
\begin{tikzpicture}[trim left=0cm,xscale=1.0,yscale=1.0]
\clip (-1.5,-9.4) rectangle (10.3,5.2);
\fill[green!50!white] (-1.5,-0.8) -- (11.0,-0.8) -- (11.0,-9.5) -- (-1.5,-9.5) -- cycle;
\fill[green!50!white] (-1.5,2.5) -- (11.0,2.5) -- (11.0,9.5) -- (-1.5,9.5) -- cycle;
\draw (-0.35,6.5) -- (-0.35,-0.35) -- (11.5,-0.35);
\draw (-0.35,-9.5) -- (-0.35,-0.35) -- (11.5,-0.35);
\draw[blue,dashed] (-0.35,-0.5) -- (10.35, -11.5);
\draw (-1,-9) node{$\vdots$};
\draw (-1,-8) node{$t=-8$};
\draw (-1,-7) node{$t=-7$};
\draw (-1,-6) node{$t=-6$};
\draw (-1,-5) node{$t=-5$};
\draw (-1,-4) node{$t=-4$};
\draw (-1,-3) node{$t=-3$};
\draw (-1,-2) node{$t=-2$};
\draw (-1,-1) node{$t=-1$};
\draw (-1,0) node{$t=0$};
\draw (-1,1) node{$t=1$};
\draw (-1,2) node{$t=2$};
\draw (-1,3) node{$t=3$};
\draw (-1,4) node{$t=4$};
\draw (-1,5) node{$\vdots$};
\draw (0.15,-0.6) node{$s=2a$};
\draw (2,-0.6) node{$s=2a+2$};
\draw (4,-0.6) node{$s=2a+4$};
\draw (6,-0.6) node{$s=2a+6$};
\draw (8,-0.6) node{$s=2a+8$};
\draw (10,-0.6) node{$\dots$};
\draw (0,0) node{$\square$};
\draw (0,2) node{$\diamond$};
\draw (2,0) node{$\square$};
\draw (2,2) node{$\diamond$};
\draw (4,0) node{$\square$};
\draw (4,2) node{$\diamond$};
\draw (6,0) node{$\square$};
\draw (6,2) node{$\diamond$};
\draw (8,0) node{$\square$};
\draw (8,2) node{$\diamond$};
\draw (10,0) node{$\dots$};
\draw (10,2) node{$\dots$};
\draw (0,-1) node{$\circ$};
\draw (2,-1) node{$\circ$};
\draw (4,-1) node{$\circ$};
\draw (6,-1) node{$\circ$};
\draw (8,-1) node{$\circ$};
\draw (10,-1) node{$\dots$};
\draw (0,-3) node{$\circ$};
\draw (2,-3) node{$\circ$};
\draw (4,-3) node{$\circ$};
\draw (6,-3) node{$\circ$};
\draw (8,-3) node{$\circ$};
\draw (10,-3) node{$\dots$};
\draw (0,5) node{$\vdots$};
\draw (2,5) node{$\vdots$};
\draw (4,5) node{$\vdots$};
\draw (6,5) node{$\vdots$};
\draw (8,5) node{$\vdots$};
\draw (0,3) node{$\circ$};
\draw (2,3) node{$\circ$};
\draw (4,3) node{$\circ$};
\draw (6,3) node{$\circ$};
\draw (8,3) node{$\circ$};
\draw (10,3) node{$\dots$};
\draw (0,-5) node{$\circ$};
\draw (2,-5) node{$\circ$};
\draw (4,-5) node{$\circ$};
\draw (6,-5) node{$\circ$};
\draw (8,-5) node{$\circ$};
\draw (10,-5) node{$\dots$};
\draw (0,-7) node{$\circ$};
\draw (2,-7) node{$\circ$};
\draw (4,-7) node{$\circ$};
\draw (6,-7) node{$\circ$};
\draw (8,-7) node{$\circ$};
\draw (10,-7) node{$\dots$};
\draw (0,-9) node{$\vdots$};
\draw (2,-9) node{$\vdots$};
\draw (4,-9) node{$\vdots$};
\draw (6,-9) node{$\vdots$};
\draw (8,-9) node{$\vdots$};
\end{tikzpicture} 
\end{minipage}
\begin{itemize}
\item In the bidegrees marked with squares (i.e., the $t=0$ row, in the even-numbered columns), we have a direct sum of copies of $\mathbb{Z}[P^{-1}]$.
\item In the bidegrees marked with diamonds (i.e., the $t=2$ row, in the even-numbered columns), we have a direct sum of copies of $\mathbb{Q}/\mathbb{Z}[P^{-1}]$.
\item The white-colored region is trivial in all other bidegrees.
\item In the green-colored regions, in each bidegree in an even-numbered column and an odd-numbered row, we have a finite abelian group (perhaps trivial, depending on $P$ and the bidegree).  Those bidegrees are marked with a circle.
\item The green-colored regions are trivial in all other bidegrees.
\end{itemize}
In particular, on and below the blue dashed line (i.e., the line $s+t=2a-1$), the $E_2$-term is concentrated in even-numbered columns and odd-numbered rows. All differentials originating below the blue-dashed line are zero for degree reasons, and no differentials originating above the blue-dashed line can hit elements below the blue-dashed line. 

Hence the elements below the blue-dashed line in the $E_2$-term survive unchanged to the $E_{\infty}$-term. The bidegrees below the blue-dashed line are precisely those which contribute, in the abutment, to $\pi_{n}(L_{KU[P^{-1}]}DX)$ with $n\geq 1-2a$. 

An analogous argument shows that there can be no nonzero differentials involving bidegrees strictly above the line $s+t = 2b+2$. Consequently the $E_{\infty}$-term coincides with the $E_2$-term in all those bidegrees which contribute in the abutment to $\pi_{n}(L_{KU[P^{-1}]}DX)$ with $n < -2b-2$. 

Depending on the attaching maps in the CW-complex $X$, there may be additive filtration jumps, so that the abutment is not simply the direct sum of the bidegrees in the $E_{\infty}$-page. However, such filtration jumps do not affect the {\em number} of elements in a given degree. Hence, if $n<-2b-2$ or $n\geq 1-2a$, then the order of the homotopy group $\pi_n(L_{KU}DX)[P^{-1}]$ is equal to the product of the orders of the groups $H^{t}(X;\mathbb{Z}[P^{-1}])\otimes_{\mathbb{Z}[1/2]}\pi_{n+t}(L_{KU[1/2]}S^0))$ for all integers $t$. The calculation of the orders of the groups $\pi_{n}(L_{KU[1/2]}S^0))$ for $n<-2$, from Theorem \ref{ravenel calc}, then yields the first claim in the statement of the theorem: if $n\geq 1-2a$, then the total order of the bidegrees $(s,t)$ below the dashed blue line contributing to $\pi_n(L_{KU}DX)[P^{-1}]$ is equal to
\begin{align*}
 \prod_{j} \left|\pi_{2j+n}(L_{KU[1/2]}S^0)\right|^{\beta_{2j}(X)}
  &= \left\{\begin{array}{ll} 
     1 &\mbox{if\ } 2\mid n \\
     \prod_j\denom(\zeta(-j - \frac{n-1}{2}))^{\beta_{2j}(X)}&\mbox{if\ } 2\nmid n\end{array}\right.
\end{align*}
up to powers of primes in $P$.
By a similar argument, if $n<-2b-2$, then 
the total order of the bidegrees $(s,t)$ below the dashed blue line contributing to $\pi_n(L_{KU[1/2]}DX)$ is equal to 
\begin{align*}
 \prod_{j} \left|\pi_{2j+n}(L_{KU[1/2]}S^0)\right|^{\beta_{2j}(X)} 
  &= \prod_{j} \left|\pi_{-2-2j-n}(L_{KU[1/2]}S^0)\right|^{\beta_{2j}(X)} \\
  &= \left\{\begin{array}{ll} 
     1 &\mbox{if\ } 2\mid n \\
     \prod_j\denom(\zeta(j + \frac{n+3}{2}))^{\beta_{2j}(X)}&\mbox{if\ } 2\nmid n\end{array}\right.
\end{align*}
up to powers of primes in $P$.

Consequently:
\begin{itemize} 
\item
$\pi_n(L_{KU[P^{-1}]}DX)$ vanishes for even integers $n$ satisfying $n\geq 1-2a$ or $n<-2b-2$. This proves the first claim.
\item If $2k-1\geq 1-2a$, then $\pi_{2k-1}(L_{KU[P^{-1}]}DX)$ is finite of order $\prod_j\denom(\zeta(1-j-k))^{\beta_{2j}(X)}$, up to factors of primes in $P$. This yields formula \eqref{spec val formula 1} in the second claim.
\item If $2k-1 < -2b-2$, then $\pi_{2k-1}(L_{KU[P^{-1}]}DX)$ is finite of order $\prod_j\denom(\zeta(j+k+1))^{\beta_{2j}(X)}$, up to factors of primes in $P$. This yields formula \eqref{spec val formula 2} in the third claim.
\end{itemize}

The product \eqref{spec val formula 1} is equal to the denominator of $\dot{\zeta}_{KU}(1-k,X)$ up to primes in $P$ and primes which occur in {\em numerators} of values of $\dot{\zeta}^{(w)}_{KU}(1-k,X)$, since such factors in numerators may plausibly cancel with factors in denominators of $\dot{\zeta}^{(w^{\prime})}_{KU}(1-k,X)$ for some weights $w^{\prime}\neq w$. By classical work of Kummer and the relationship between Riemann zeta special-values and Bernoulli numbers, it is precisely the irregular primes which occur as numerators of special values of $\zeta(s)$ at negative integers\footnote{It is possible to improve on this theorem by showing that under some conditions on $X$, there is no cancellation between factors in numerators of special values and factors in denominators of special values. The relevant tool here is the Herbrand-Ribet theorem \cite{MR0419403}, which, for a given irregular prime $p$, gives an algebraic characterization of {\em which} special values of $\zeta(s)$ will have numerator divisible by $p$. These methods are highly compatible with those in this paper, but the resulting conditions on $X$ are more meticulous and technical to state. We leave off that direction of investigation for a later time.}. Formulas \eqref{spec val formula 1a}  and \eqref{spec val formula 2a} follow.
\end{proof}

\begin{corollary}
Let $V$ be a smooth projective cellular variety over $\mathbb{Q}$, with Hasse-Weil zeta-function $\zeta_V(s)$. Let $\mathbb{C}(V)$ denote the complex analytic space associated to $V$. Suppose that the cohomology $H^*(\mathbb{C}(V);\mathbb{Q})$ is concentrated in even degrees. Let $P$ be the set of primes 
\[ P = \{ 2\} \cup \left\{ p \mbox{\ prime}: H^*(\mathbb{C}(V);\mathbb{Z})\mbox{\ has\ nontrivial\ $p$-torsion}\right\} .\]
Then, for each positive odd integer $2k-1$, the denominator of the Hasse-Weil zeta-value $\zeta_V(1-k)$ is equal to the order of the $KU$-local stable homotopy group $\pi_{2k-1}\left(L_{KU}D(\mathbb{C}(V))\right)$ of the Spanier-Whitehead dual $D(\mathbb{C}(V))$, up to powers of irregular primes and primes in $P$.
\end{corollary}
\begin{proof}
Consequence of Example \ref{cellular variety example} and Theorem \ref{main spec vals thm}. 
\end{proof}

\begin{remark}\label{cancellation remark 1}
The statement of Theorem \ref{main spec vals thm} suggests that, at irregular primes, there may be a discrepancy between the denominator of $\dot{\zeta}_{KU}(1-k,X)$ and the product \eqref{spec val formula 1}, i.e., the order of $\pi_{-1-2k}(L_{KU}X)$. This indeed can happen. As an amusing example, let $X$ be the cofiber of the map $S^{1355}\rightarrow S^0$ given by the $227$-primary $\alpha_3\in \pi_{1355}(S^0)$, i.e., a generator for the $227$-torsion in the third stable stem in which nonzero $227$-torsion appears\footnote{To be clear: the numbers $227$ and $691$ appearing in this example are indeed prime.}. Then $\pi_{-25}L_{KU}X$ is cyclic of order equal to 
\begin{align*}
 \denom(\zeta(-11))\cdot \denom(\zeta(-689))
  &= 32760 \cdot 387923085396 \\
  &= 2^5\cdot 3^4\cdot 5\cdot 7^2\cdot 11\cdot 13\cdot 31\cdot 47\cdot 139\cdot 691
\end{align*}
up to a power of $2$,
while the denominator of $\dot{\zeta}_{KU}(-11,X)$ is equal to
\begin{align*}
 \denom(\zeta(-11)\cdot \zeta(-689))
  &= 2^5\cdot 3^4\cdot 5\cdot 7^2\cdot 11\cdot 13\cdot 31\cdot 47\cdot 139.
\end{align*}
The discrepancy is because the prime $691$ is irregular: $\zeta(-11) = \frac{691}{32760}$, and the factor of $691$ in the numerator of $\zeta(-11)$ cancels with the factor of $691$ in the denominator of $\zeta(-689)$. 
\end{remark}

Here is one more comment on the ``crudeness'' of the provisional $KU$-local zeta-function.
\begin{definition}
Suppose $E$ is a spectrum such that $\pi_n(L_ES^0)$ is finite for all $n<<0$.
Let say that finite spectra $X$ and $Y$ are {\em $E$-locally numerically equivalent} if there exists some integer $N$ such that
\begin{itemize}
\item for all $n<N$, the abelian groups $\pi_n(L_EX)$ and $\pi_n(L_EY)$ are each finite,
\item and for all $n<N$, $\left|\pi_n(L_EX)\right|=\left|\pi_n(L_EY)\right|$.
\end{itemize}
\end{definition}
If $X$ and $Y$ are $E$-locally equivalent (in the sense of Bousfield \cite{MR551009}), then $X$ and $Y$ are clearly also $E$-locally numerically equivalent. However, the converse is not true: we see from Theorem \ref{main spec vals thm} that, for finite spectra $X$ and $Y$ with torsion-free cohomology concentrated in even degrees, all that is necessary for $X$ and $Y$ to be numerically $KU$-locally equivalent is for $X$ and $Y$ to be {\em rationally} equivalent. Of course there are many such examples: for example, the cofiber of $\alpha_1\in \pi_{2p-3}(S^0)$ is a finite spectrum $X$ whose cohomology is torsion-free and concentrated in even degrees. This spectrum $X$ is {\em not} $KU$-locally equivalent to the wedge sum $S^0\vee S^{2p-2}$, and yet $X$ and $S^0\vee S^{2p-2}$ are rationally equivalent, and $KU$-locally numerically equivalent. One can replace the role of $\alpha_1$ here with any one of the divided alpha elements in the stable stems, and arrive at the same conclusion. 

All those examples are merely $2$-cell complexes, and of course attaching more cells to kill $KU$-locally nontrivial torsion elements in $\pi_*$ in odd degrees yields many more examples. Our point is simply that there is an ample supply of $KU$-local stable homotopy types to which Theorem \ref{main spec vals thm}, and more generally the results of this section, apply.

\begin{remark}
Of course the group $\hat{\mathbb{Z}}_{p}^{\times}$ of Adams operations on $p$-complete $K$-theory is abstractly isomorphic to the Galois group $\Gal(\mathbb{Q}(\zeta_{p^{\infty}})/\mathbb{Q})$. Sullivan's approach to the Adams conjecture \cite{MR0442930} involved producing a {\em particular} such isomorphism. Consequently, if we take the $K$-theory of a finite CW-complex and tensor it with the complex numbers, the resulting complex vector space $KU^0(X)\otimes_{\mathbb{Z}}\mathbb{C}$ is a representation of $\Gal(\mathbb{Q}(\zeta_{p^{\infty}})/\mathbb{Q})$, hence by Kronecker-Weber, a direct sum of degree $1$ representations of the absolute Galois group $\Gal(\overline{\mathbb{Q}}/\mathbb{Q})$. The provisional zeta function $\dot{\zeta}_{KU}(s,X)$ agrees with the $L$-function of that Galois representation.

It is not hard to formulate the ideas in this section as a kind of ``class field theory of spectra.'' Recall that the local Langlands correspondence for $GL_n$ establish an $L$-function-preserving bijection between certain irreducible representations of $GL_n(K)$ and certain irreducible degree $n$ representations of $\Gal(\overline{K}/K)$. The $n=1$ case amounts to class field theory: the Dirichlet characters (or more generally, Hecke characters) on the ``automorphic side'' of the correspondence are matched up with degree $1$ Galois representations on the ``spectral side'' of the correspondence. The Dirichlet (or Hecke) $L$-functions on the automorphic side are equal to the Artin $L$-functions on the spectral side. 

The complex $K$-theory of a finite CW-complexes, tensored up to $\mathbb{C}$, yields representations of the abelianized Galois group $\Gal(\overline{\mathbb{Q}}/\mathbb{Q})$. One can match up these Galois representations with Dirichlet characters in a way that preserves the $L$-functions. This is fine, but in the state described in this section, it is not a very good theory. There are simply too few Galois representations which arise as $KU^0(X)\otimes_{\mathbb{Z}}\mathbb{C}$ for $X$ a finite CW-complex.

Any genuinely {\em useful} class field theory of spectra ought to involve many more Galois representations than the ones arising in this section. In the next section, we give a more general construction of Galois representations associated to finite CW-complexes. That construction yields a much richer supply of Galois representations and corresponding $L$-functions. The essential idea is not merely to tensor the complex $K$-theory $KU^*(X)$ with $\mathbb{C}$, which destroys all information about torsion elements in $KU^*(X)$. Instead, we must find a way to use the {\em torsion} in $K$-theory in the process of defining $KU$-local $L$-functions and zeta-functions.
\end{remark}

\section{The $KU$-local zeta-function of a finite CW-complex with squarefree torsion in $K$-theory.}
\label{section on even and odd}

\subsection{Defining the $KU$-local torsion $L$-function and $KU$-local zeta-function.}
\label{Defining the... torsion...}

In this section, we will modify and improve the ``provisional $KU$-local zeta-function'' $\dot{\zeta}_{KU}(s,X)$ of a finite CW-complex $X$. 
%
It is now time to restrict the level of generality, 
in order to tidy up the theory and make it far more usable. The theorems proven in \cref{Torsion-free...concentrated in even dimensions} applied only to finite CW-complexes whose cohomology (hence also $K$-theory), after inverting $P$, is concentrated in even degrees. Earlier, in \cref{The provisional...}, we defined the provisional $KU$-local zeta-function in such a way so as not to disallow CW-complexes with $K$-theory in odd degrees, in order to maintain similarity with the classical story of Hasse-Weil zeta-functions. But, as the reader can see from \cref{Torsion-free...concentrated in even dimensions}, when the cohomology $H^*(X; \mathbb{Z}[P^{-1}])$ is not concentrated in even degrees, we were able to prove much less about $\dot{\zeta}_{KU}(s,X)$. 

It will simplify our theory dramatically if, from now on, we only work with CW-complexes whose cohomology is concentrated in even degrees, after inverting a set of primes $P$. Beginning in this section, we restrict to that level of generality\footnote{For the reader who may find this restricted generality disappointing, we remark that it is probably sensible to think about {\em two} $KU$-local zeta-functions for a given finite CW-complex, one for the even $K$-groups and one for the odd $K$-groups, and to simply treat them as a pair, rather than as a single function. Something roughly similar is done already in the study of $p$-adic $L$-functions, where one begins with a classical $L$-function $L(s)$ and $p$-adically interpolates its values $L(-a),L(1-p-a),L(2-2p-a),L(3-3p-a), \dots$ {\em separately} for each individual residue class $a$ modulo $p-1$.}.


In Definition \ref{def of ku-local l and zeta}, below, we will give our improved version of the ``provisional $KU$-local zeta-function.'' We will also need an auxiliary definition of something like a ``$KU$-local Dirichlet $L$-function.'' The idea is to modify our definition of the provisional $KU$-local zeta-function with some extra factors that keep track of Adams operations on torsion in the $K$-theory of the CW-complex $X$. We cannot simply take the determinant of the action of the Adams operations on torsion in $K$-theory, since we would not get well-defined complex numbers that way. Instead, we must think of the various ways of embedding the torsion in $KU^0(X)$ into the complex numbers. That is, we must consider the various {\em complex representations} of the torsion subgroup of $KU^0(X)$. To each such representation (satisfying some hypotheses) we will associate an $L$-function.

For a finite CW-complex $X$ with torsion-free cohomology (hence torsion-free $K$-theory) concentrated in even degrees, the filtration of $K$-theory by the skeleta of $X$ coincides with the filtration of $K$-theory defined by the eigenvalues of the Adams operations on $KU^*(X)$. The agreement of these two filtrations is lost when one begins to study torsion in $K$-theory: the filtration by Adams eigenvalues is distinct from the skeletal filtration. It is the skeletal filtration which has good properties, when considering torsion in $K$-theory. Here is the relevant definition:
\begin{definition}\label{def of weight of rep}
Let $P$ be a set of primes\footnote{Here is an explanation of the role of $P$ here and throughout this section. Later, in Theorem \ref{main spec vals thm 2}, we will be able to prove good properties of a certain zeta-function associated to a finite CW-complex $X$ with cohomology concentrated in even degrees, under a hypothesis which is a bit weaker than asking that the torsion subgroup of $KU^0(X)$ has square-free order. We might reasonably want to apply these methods and results to finite CW-complexes which fail to satisfy that hypothesis. Since $X$ is a {\em finite} CW-complex, there is some finite set of primes $P$ such that, after inverting the primes in $P$, $X$ {\em does} satisfy the hypotheses of Theorem \ref{main spec vals thm 2}. So the role of $P$ throughout this section is that it is a set of ``bad primes'' which we will invert, so that the results of this section can be applied to {\em any} finite CW-complex $X$.}, including $2$. Let $X$ be a finite CW-complex with cohomology $H^*(X;\mathbb{Z}[P^{-1}])$ concentrated in even degrees. 
 Let $\rho$ be a complex representation of the torsion subgroup of $KU^0(X)[P^{-1}]$. 
\begin{itemize}
\item
We will say that $\rho$ has {\em skeletal weight $w$} if $w$ is the least integer $n$ such that the composite map
\[ \tors KU^0(X/X^{2n}) \rightarrow \tors KU^0(X)[P^{-1}] \stackrel{\rho}{\longrightarrow} GL(V) \]
is trivial. Here $X^{2n}$ denotes the $2n$-skeleton of $X$.
\item 
We will say that {\em the skeletal filtration of $\tors KU^0(X)[P^{-1}]$ splits additively} if each of the subgroup inclusions in the skeletal filtration 
\begin{equation}\label{skel filt 1}\xymatrix{ 
 \vdots \ar@{^{(}->}[d] \\
 \ker\left( \tors KU^0(X)[P^{-1}]\rightarrow \tors KU^0(X^6)[P^{-1}] \right) \ar@{^{(}->}[d] \\
 \ker\left( \tors KU^0(X)[P^{-1}]\rightarrow \tors KU^0(X^4)[P^{-1}] \right) \ar@{^{(}->}[d] \\
 \ker\left( \tors KU^0(X)[P^{-1}]\rightarrow \tors KU^0(X^2)[P^{-1}] \right) \ar@{^{(}->}[d] \\
 \vdots \ar@{^{(}->}[d] \\
 \tors KU^0(X)[P^{-1}] }\end{equation}
of $\tors KU^0(X)[P^{-1}]$ is a split monomorphism of abelian groups.
We will say that {\em the skeletal filtration of $\tors KU^0(X)[P^{-1}]$ splits completely} if, after completing at each prime $\ell$,
each of the subgroup inclusions in \eqref{skel filt 1} is a split monomorphism of $\hat{\mathbb{Z}}_{\ell}[\hat{\mathbb{Z}}_{\ell}^{\times}]$-modules. 
\item 
Regardless of whether the skeletal filtration of $\tors KU^0(X)[P^{-1}]$ splits, we will write $\tors^{(w)} KU^0(X)[P^{-1}]$ for the skeletal filtration $w$ subquotient of $\tors KU^0(X)[P^{-1}]$, i.e.,
\begin{align*}
 \tors^{(w)} KU^0(X)[P^{-1}]
  &= \frac{\ker\left( \tors KU^0(X)[P^{-1}]\rightarrow \tors KU^0(X^{2w})[P^{-1}] \right)}{\ker\left( \tors KU^0(X)[P^{-1}]\rightarrow \tors KU^0(X^{2w+2})[P^{-1}] \right)}.
\end{align*}
\end{itemize}
\end{definition}

Now we can define an Euler product which, again, cleaves as closely as possible to the Hasse-Weil Euler product \eqref{hasse-weil def 1}, but this time, our Euler product will ``pay attention'' to torsion in $K$-theory. Let $P$ again be a set of primes, including $2$. Suppose that $X$ is a finite CW-complex whose cohomology $H^*(X;\mathbb{Z}[P^{-1}])$ is concentrated in even degrees. 
Furthermore, suppose 
that, for each integer $w$, the order of $\tors^{(w)} KU^0(X)[P^{-1}]$ is square-free. Let $n_w$ denote the order of $\tors^{(w)} KU^0(X)[P^{-1}]$. The group of units $\mathbb{Z}/n_w^2\mathbb{Z}^{\times}$ decomposes canonically as the product $\prod_{\ell\mid n_w} \mathbb{Z}/{\ell}^2\mathbb{Z}^{\times}$ taken over all the primes $\ell$ dividing $n_w$. The group $\tors^{(w)}KU^0(X)[P^{-1}]$ is non-canonically isomorphic to the product, over all such primes $\ell$, of the $\ell$-Sylow subgroup of $\mathbb{Z}/\ell^2\mathbb{Z}^{\times}$. Write $\Syl(n_w)$ for that product of $\ell$-Sylow subgroups, and choose an isomorphism $i_w: \Syl(n_w) \stackrel{\cong}{\longrightarrow}\tors^{(w)} KU^0(X)[P^{-1}]$. 

Now, given a complex representation $\rho$ of $\tors^{(w)} KU^0(X)[P^{-1}]$, we have group homomorphisms
\begin{align*}
 \Syl(n_w) 
  &\stackrel{i_w}{\longrightarrow} \tors^{(w)} KU^0(X)[P^{-1}] \\
  &\stackrel{\Psi^p}{\longrightarrow} \tors^{(w)} KU^0(X)[P^{-1}] \\
  &\stackrel{\rho}{\longrightarrow} GL(V)
\end{align*}
where $p$ is any prime not dividing $n_w$. 
The resulting homomorphism $\rho\circ \Psi^p\circ i_w$ extends canonically to a group homomorphism 
\begin{align*}
 \Psi_{\rho,p,i_w}: \mathbb{Z}/n_w\mathbb{Z}^{\times} &\rightarrow GL(V)\end{align*}
which
\begin{itemize} 
\item agrees with $\rho\circ \Psi^p\circ i_w$ on the summand $\Syl(n_w)$ of $\mathbb{Z}/n_w\mathbb{Z}^{\times}$, 
\item and is trivial on the complementary summand 
of $\Syl(n_w)$ in $\mathbb{Z}/n_w\mathbb{Z}^{\times}$.
\end{itemize}

For any such prime $p$, the group $\mathbb{Z}/n_w\mathbb{Z}^{\times}$ furthermore has a particular element named $p$---simply the integer $p$ in $\mathbb{Z}$, regarded as a residue class in $\mathbb{Z}/n_w\mathbb{Z}^{\times}$. We may evaluate $\Psi_{\rho,p,i_w}$ at $p$ to get an element of $GL(V)$.
Consider the Euler product 
\begin{align}\label{euler product tors 0}
 \prod_{p} \det\left(\id - p^{w-s}\Psi_{\rho,p,i_w}(p)\right)^{-1}
\end{align}
taken over all primes $p$ which do not divide $n_w$.   
The Euler product \eqref{euler product tors 0} mimics the Euler product of the Dirichlet $L$-function of a Dirichlet character $\chi$,
\[ \prod_p \left( 1 - p^{-s} \chi(p)\right)^{-1}\]
as well as the Euler product of the Artin $L$-function of a Galois representation $\rho: \Gal(F/\mathbb{Q})\rightarrow GL(V)$,
\[ \prod_p \det\left( \id - p^{-s}\rho\left(\Fr\mid_{V^{I_p}}\right)\right)^{-1},\]
where $\Fr$ is a lift of the Frobenius element in $\Gal((\mathcal{O}_F/\mathfrak{p})/\mathbb{F}_p)$ to an element of $\Gal(F/\mathbb{Q})$, and $I_p$ is the $p$-inertia subgroup of $\Gal(F/\mathbb{Q})$. 
The power $w-s$, rather than $-s$, in \eqref{euler product tors 0} simply arranges for the skeletal weight $w$ torsion in $K$-theory to contribute to weight $w$ factors in the $L$-function. In the torsion-free case (as in Definition \ref{def of euler product}), this was unnecessary, since these weights all agree automatically: the action of Adams operations on rational $K$-groups naturally recovers the skeletal weight, by the relationship between the skeletal filtration and the filtration by Adams eigenvalues, discussed before Definition \ref{def of weight of rep}. Since that relationship is lost when one considers the torsion in $K$-theory, the $L$-factors coming from torsion in $K$-theory must be put into the correct weight ``by hand,'' i.e., by having a factor of $p^{w-s}$ rather than $p^{-s}$ in \eqref{euler product tors 0}.

Since $\mathbb{Z}/n_w\mathbb{Z}^{\times}$ is abelian, its representations split as direct sums of one-dimensional representations. For such a one-dimensional representation $\pi$, $\Psi_{\pi,p,i_w}$ is a homomorphism
\begin{align*}
  \mathbb{Z}/n_w\mathbb{Z}^{\times} &\rightarrow \mathbb{C}^{\times},\end{align*}
i.e., a Dirichlet character of modulus $n_w$. 
It is not necessarily the case that the Euler product \eqref{euler product tors 0} is the Dirichlet $L$-function of any single Dirichlet character, though. The trouble is that, due to the effect of the Adams operations, for distinct primes $p_1$ and $p_2$ it may happen that the $p_1$-local Euler factor in \eqref{euler product tors 0} is the $p_1$-local Euler factor of one Dirichlet character, while the $p_2$-local Euler factor in \eqref{euler product tors 0} is the $p_2$-local Euler factor of some {\em other} Dirichlet character.

The solution is simply to take the product over {\em sufficiently many} representations: the Euler product ought not to merely be a product over prime {\em numbers,} but also a product over prime {\em representations.} 
\begin{definition}
Let $G$ be a finite cyclic group. An representation of $G$ is {\em prime} if it is irreducible and induced from a nontrivial representation of a subgroup of $G$ of prime order.
\end{definition}
Now consider the Euler product
\begin{align}\label{euler product tors 1}
 \prod_{w\in \mathbb{Z}}\prod_{\rho_w}\prod_{p} \det\left(\id - p^{w-s}\Psi_{\rho_w,p,i_w}(p)\right)^{-1}
\end{align}
where:
\begin{itemize}
\item the product $\prod_{\rho_w}$ ranges over all the {\em prime} representations of $\tors^{(w)}KU^0(X)[P^{-1}]$,
\item and the product $\prod_p$ is taken over all primes $p$ not dividing the order of $\tors KU^0(X)[P^{-1}]$.
\end{itemize}
Of course the product $\prod_{w\in\mathbb{Z}}$ in \eqref{euler product tors 1} is finite, since $X$ is a finite CW-complex, so $\tors^{(w)}KU^0(X)[P^{-1}]$ is trivial for all but finitely many skeletal weights $w$.

We reiterate our assumptions so far in this section: 
\begin{assumptions}\label{running assumptions}
\begin{itemize}
\item $P$ is a set of primes, including $2$,
\item $X$ is a finite CW-complex with cohomology $H^*(X;\mathbb{Z}[P^{-1}])$ concentrated in even dimensions, 
\item and, for each integer $w$, the skeletal weight $w$ subquotient of $\tors KU^0(X)[P^{-1}]$ has square-free order $n_w$.
\end{itemize}
\end{assumptions}

We also adopt the following notation: given a Dirichlet character $\chi$, we write $\tilde{\chi}$ for its associated primitive character, i.e., $\tilde{\chi}$ is a Dirichlet character of modulus equal to the conductor of $\chi$, and $\chi$ is induced up from $\tilde{\chi}$.

With those assumptions and that notation, we have the following result:
\begin{prop}\label{definedness of tors zeta}
The Euler product \eqref{euler product tors 1} is equal to
\begin{equation}\label{euler product formula 1} \prod_{w\in\mathbb{Z}} \prod_{\chi\in \Dir_{prime}(n_w^2)}L(s-w,\tilde{\chi}),\end{equation}
where $\Dir_{prime}(n_w^2)$ is the set of Dirichlet characters of modulus $n_w^2$ which
\begin{itemize}
\item have conductor equal to $\ell^2$ for some prime divisor $\ell$ of $n_w$, and
\item are trivial on the complementary summand of $\Syl_{\ell}(\mathbb{Z}/n_w^2\mathbb{Z}^{\times})\subseteq \Syl(n_w)$ in $\mathbb{Z}/n_w^2\mathbb{Z}^{\times}$.
\end{itemize} 

Consequently the Euler product \eqref{euler product tors 1} does not depend on the choice of $i_w$, and it converges absolutely for all $s\in\mathbb{C}$ such that $\mathfrak{Re}(s)>1+c$, where $c$ is\footnote{A simple upper bound for $c$ is the greatest integer $N$ such that $H^{2N}(X;\mathbb{Z}[P^{-1}])$ is nontrivial. Hence, if we call the latter integer $b$, then \eqref{euler product tors 1} converges absolutely for all $s\in\mathbb{C}$ such that $\mathfrak{Re}(s)>1+b$.} the greatest integer $N$ such that $\tors KU^0(X)[P^{-1}]$ has a nontrivial summand of skeletal weight $N$. Furthermore, \eqref{euler product tors 1} analytically continues to a meromorphic function on the complex plane.
\end{prop}
\begin{proof}
Fix an integer $w$. 
Let $p$ be a prime not dividing $n_w$.
For each prime $\ell$ not dividing $n_w$,
we know from the Atiyah-Hirzebruch spectral sequence 
\begin{align*}
H^*(X;\hat{KU}^*_{\ell}(S^0)[P^{-1}]) &\Rightarrow \hat{KU}^{*}_{\ell}(X)[P^{-1}]
\end{align*}
that the $\hat{\mathbb{Z}}_{\ell}[\hat{\mathbb{Z}}_{\ell}^{\times}]$-module $\tors^{(w)}\hat{KU}^0_{\ell}(X)[P^{-1}]$ is a subquotient of $\hat{KU}^{0}_{\ell}(X^{2w}/X^{2w-1})$, i.e., the $\ell$-complete $K$-theory of a wedge of $2w$-dimensional spheres, on which the Adams operation $\Psi^p$ acts simply by multiplication by $p^w$. The Euler product \eqref{euler product tors 1} is taken only over those primes $p$ which do not divide $n_w$, so the action of $\Psi^p$ on $\tors^{(w)}KU^0(X)[P^{-1}]$ is by an automorphism. This automorphism depends on $p$ and on $w$, but of course it does not depend on $\rho_w$. Hence, as $\rho$ ranges over the set of prime representations of $\tors^{(w)} KU^0(X)[P^{-1}]$ (a set with $\phi(n_w)$ elements), the representations $\Psi_{\rho,p,i_w}$ range over the $\phi(n_w)$ elements of $\Dir_{prime}(n_w^2)$. That is, given an element $\chi\in \Dir_{prime}(n_w^2)$, there exists precisely one prime representation $\rho_w$ such that $\chi = \Psi_{\rho_w,p,i_w}$. 

Now consider the $p$-local Euler factor $(1 - p^{w-s}\tilde{\chi}(p))^{-1}$ in the Dirichlet $L$-series of $\tilde{\chi}$ evaluated at $w-s$. This $L$-factor occurs precisely once as an $L$-factor in \eqref{euler product tors 1}, as $\det\left(\id - p^{w-s}\Psi_{\rho_w,p,i_w}(p)\right)^{-1}$ for precisely that prime representation $\rho_w$ such that $\chi = \Psi_{\rho_w,p,i_w}$.
The product formula \eqref{euler product formula 1} 
follows immediately.
\end{proof}
\begin{remark}
Choose a Dirichlet character $\chi$ of conductor $\ell^2$ appearing in \eqref{euler product formula 1}. Then $\chi$ is trivial on the complementary summand $\mathbb{F}_{\ell}^{\times}$ of the $\ell$-Sylow subgroup of $\mathbb{Z}/\ell^2\mathbb{Z}^{\times}$. This is precisely the {\em opposite} of the behavior of the mod $\ell$ cyclotomic character and its nontrivial powers.

The Dirichlet $L$-functions considered in Mitchell's paper \cite{MR2181837} are $L$-functions only of powers of cyclotomic characters. For that reason, the relationships between $K(1)$-local homotopy theory and special values of $L$-functions considered in \cite{MR2181837} are orthogonal to the relationships considered in this paper. (The paper \cite{MR2181837} also is about $K(1)$-local algebraic $K$-theory, rather than localizations of finite CW-complexes. This difference is important and fundamental, as described in a footnote in \cref{The main ideas and results.}.)
\end{remark}

\begin{definition}\label{def of ku-local l and zeta}
Let $P$ be a set of primes and let $X$ be a finite CW-complex satisfying Assumptions \ref{running assumptions}. 
\begin{itemize}
\item
The {\em torsion $KU[P^{-1}]$-local $L$-function of $X$}, written $L_{\tors KU[P^{-1}]}(s,X)$, is the meromorphic function on $\mathbb{C}$ given by the analytic continuation of the Euler product $\prod_{w\in \mathbb{Z}}\prod_{\rho_w}\prod_{p} \det\left(\id - p^{w-s}\Psi_{\rho_w,p,i_w}(p)\right)^{-1}$ from \eqref{euler product tors 1}.
\item 
The {\em $KU[P^{-1}]$-local zeta-function of $X$}, written $\zeta_{KU[P^{-1}]}(s,X)$, is the product of the provisional $KU$-local zeta-function of $X$ (defined in Definition \ref{def of provisional ku-local zeta}) with the torsion $KU[P^{-1}]$-local $L$-function of $X$:
\begin{align*}
 \zeta_{KU[P^{-1}]}(s,X) &= \dot{\zeta}_{KU}(s,X)\cdot L_{\tors KU[P^{-1}]}(s,X).
\end{align*}
\item 
Given an integer $w$, we will write $L_{\tors KU[P^{-1}]}^{(w)}(s,X)$ for the weight $w$ factor \[\prod_{\rho_w}\prod_{p} \det\left(\id - p^{w-s}\Psi_{\rho_w,p,i_w}(p)\right)^{-1}\] of $L_{\tors KU[P^{-1}]}(s,X)$. We will write $\zeta_{KU[P^{-1}]}^{(w)}(s,X)$ for the weight $w$ factor \[\dot{\zeta}_{KU[P^{-1}]}^{(w)}(s,X)\cdot L_{\tors KU[P^{-1}]}^{(w)}(s,X)\] of $\zeta_{KU[P^{-1}]}(s,X)$.
\item 
Finally, given an integer $w$ and a prime divisor $\ell$ of the order $n_w$ of $\tors^{(w)} KU^0[P^{-1}]$, we write 
$L_{\tors KU[P^{-1}]}^{(w,\ell)}(s,X)$ for the factor 
\begin{align}
\label{def 340589jf}
  L_{\tors KU[P^{-1}]}^{(w,\ell)}(s,X)
   &= \prod_{\chi}L(s-w,\tilde{\chi}),
\end{align}
of $L_{\tors KU[P^{-1}]}^{(w)}(s,X)$, where the product \eqref{def 340589jf} is taken over all the characters $\chi\in\Dir_{prime}(n_w^2)$ of conductor equal to $\ell^2$. 
\end{itemize}
\end{definition}
As a consequence of Definition \ref{def of provisional ku-local zeta} and Proposition \ref{definedness of tors zeta}, we have
\begin{align}
\label{eq 2309} \zeta_{KU[P^{-1}]}(s,X)
  &= \prod_{w\in\mathbb{Z}} \left( \zeta(s-w)^{\dim_{\mathbb{Q}}H_{2w}(X;\mathbb{Q})}\cdot \prod_{\chi\in \Dir_{prime}(n_w^2)}L(s-w,\chi)\right),
\end{align}
where again $n_w$ is the order of the skeletal weight $w$ summand of $\tors KU^0(X)[P^{-1}]$.

\subsection{Functional equations for $L_{\tors KU}(s,X)$ and $\zeta_{KU}(s,X)$.}
\label{Functional equations for tors}
One can use the functional equation for Dirichlet $L$-functions to prove a functional equation for $L_{\tors KU[P^{-1}]}(s,X)$, as follows.
The classical completed Dirichlet $L$-function of an even\footnote{A Dirichlet character $\chi$ is {\em even} if $\chi(-1) = 1$. All Dirichlet characters whose $L$-functions occur as factors in $L_{\tors KU[P^{-1}]}(s,X)$ are even.} primitive Dirichlet character of modulus $\ell$, $\Lambda(s,\chi) := (\pi/\ell)^{-s/2}\Gamma(s/2) L(s,\chi)$, satisfies the functional equation $\Lambda(s,\chi) = W(\chi)\Lambda(1-s,\bar{\chi})$ (see for example 8.5 of \cite{MR3290245}).
Here $W(\chi)$ is the {\em root number} of the Dirichlet character $\chi$, defined as $W(\chi) := \frac{\tau(\chi)}{\sqrt{\ell}}$,
where $\tau(\chi)$ is the Gauss sum $\tau(\chi) = \sum_{m=1}^{\ell} \chi(m) e^{m/\ell}\in\mathbb{C}$.

Let $\hat{L}_{\tors KU[P^{-1}]}(s,X)$ denote the product 
$\prod_{w\in\mathbb{Z}} \prod_{\chi\in \Dir_{prime}(n_w^2)}\Lambda(s-w,\tilde{\chi})$ of the completed Dirichlet $L$-functions of the characters $\chi$ appearing in the factorization \eqref{euler product formula 1} of $L_{\tors KU[P^{-1}]}(s,X)$. 
Since $W(\chi)W(\overline{\chi}) = 1$, the product of the root numbers $\prod_{\chi\in \Dir_{\mathbb{Q}(\zeta_{\ell^2})[\ell]^*}}W(\chi)$ must be $\pm 1$, yielding a functional equation
\begin{align}\label{func eq tors 1}
 \hat{L}_{\tors KU[P^{-1}]}(s,X) 
  &= \pm \hat{L}_{\tors KU[P^{-1}]}(1-s,\Sigma^{-1} DX)
\end{align}
for $\hat{L}_{\tors KU[P^{-1}]}(s,X)$. 

Aside from the $\pm$ in \eqref{func eq tors 1} arising from the root number, the other obvious difference between the functional equation \eqref{func eq tors 1} for $\hat{L}_{\tors KU[P^{-1}]}(s,X)$ and the functional equation \eqref{func eq 2a} for $\hat{\dot{\zeta}}_{KU}(s,X)$ is the presence of the desuspension $\Sigma^{-1}$. The functional equation \eqref{func eq tors 1} relates the completed $KU$-local torsion $L$-functions of $X$ and of the desuspended Spanier-Whitehead dual $DX$, while the functional equation \eqref{func eq 2a} relates the completed $KU$-local provisional zeta-functions of $X$ and of the Spanier-Whitehead dual $DX$, {\em without} any desuspension. 

The reason for this difference is the shape of the universal coefficient sequence in the complex $K$-theory of a finite CW-complex $X$ (originally \cite{anderson}, but see \cite{MR0388375} for a published source, or Theorem IV.4.5 of \cite{MR1417719} for a useful wide generalization):
\begin{equation}\label{ucs 1} 0 \rightarrow \Ext^1_{\mathbb{Z}}(KU_{n-1}(X),\mathbb{Z}) \rightarrow KU^n(X) \rightarrow \hom_{\mathbb{Z}}(KU_{n}(X),\mathbb{Z})\rightarrow 0.\end{equation}
The left-hand term in \eqref{ucs 1} is torsion and depends only on the torsion in $KU_{n-1}(X)\cong KU^{1-n}(DX)$, while 
the right-hand term in \eqref{ucs 1} is torsion-free and depends only on the torsion-free part of $KU_{n}(X)\cong KU^{-n}(DX)$. Consequently the effect of Spanier-Whitehead dualization in $K$-theory is that there is a natural degree shift of $1$ in the $K$-theoretic torsion, while the torsion-free part of $K$-theory does {\em not} get shifted in this way\footnote{Because of this difference, the author knows of no single functional equation for the product $\hat{\zeta}_{KU[P^{-1}]}(s,X)$ of $\hat{\dot{\zeta}}_{KU}(s,X)$ and $\hat{L}_{\tors KU[P^{-1}]}(s,X)$. One would like to find some kind of dualization functor $D^{\prime}$ on finite CW-complexes, which has the same effect as the Spanier-Whitehead dualization $D$ on the torsion-free part of $K$-theory, but which, unlike $D$, does not introduce a degree shift in the torsion in $K$-theory. Then one would have a nice functional equation $\hat{\zeta}_{KU[P^{-1}]}(s,X) = \pm \hat{\zeta}_{KU[P^{-1}]}(1-s,D^{\prime}X)$. The author would be pleased to learn of a way to do this. Perhaps it is possible by modifying $D$ in roughly the way that Anderson \cite{anderson} modified the Brown-Comenetz dualization functor $I$.}. The shift in the $K$-theoretic torsion is the reason the functional equation for the $KU$-local torsion $L$-function must relate $X$ and $\Sigma^{-1} DX$, rather than $X$ and $DX$.

\subsection{Special values of the $KU$-local zeta-function.}

Finally, we had better see how these constructions relate to orders of homotopy groups, or all these definitions are worth very little. Recall that Theorem \ref{main spec vals thm} expressed orders of $KU$-local stable homotopy groups of some finite CW-complexes in terms of special values of provisional $KU$-local zeta-functions, {\em away from} $2$ and the primes dividing the order of the torsion subgroup of $KU^0(X)$. Theorem \ref{main spec vals thm 2} extends Theorem \ref{main spec vals thm} to those primes $p$ such that the order of the torsion subgroup of $KU^0(X)$ is divisible by $p$, but not $p^2$: 
\begin{theorem}\label{main spec vals thm 2}
Let $P$ be a set of primes with $2\in P$, and let $X$ be a finite CW-complex satisfying Assumptions \ref{running assumptions}. 
Let $a,b$ be the least and greatest integers $n$, respectively, such that $H^{2n}(X;\mathbb{Z}[P^{-1}])$ is nontrivial.
Then the following conditions are equivalent:
\begin{itemize}
\item The skeletal filtration of $\tors KU^0(X)[P^{-1}]$ splits completely.
\item 
For all odd integers $2k-1$ satisfying $2k-1> 1-2a$, the $KU$-local stable homotopy group $\pi_{2k-1}(L_{KU}DX)$ is finite, and up to powers of primes in $P$, its order is equal to the product\footnote{To be clear, the product $\prod_{\ell\mid n_w}$ in \eqref{product 18a} is taken over all {\em prime} divisors $\ell$ of the order $n_w$ of $\tors^{(w)} KU^0(X)[P^{-1}]$.}
\begin{equation}\label{product 18a} \prod_{w\in\mathbb{Z}}\left( \denom\left(\dot{\zeta}^{(w)}_{KU}(1-k,X)\right)\cdot \prod_{\ell\mid n_w}\denom\left( L^{(w,\ell)}_{\tors KU[P^{-1}]}(1-k,X)\right)\right)\end{equation}
of denominators of the isoweight factors in $\zeta_{KU}(s,X)$.
\item 
For all odd integers $2k-1$ satisfying $2k-1< -2b-3$, the $KU$-local stable homotopy group $\pi_{2k-1}(L_{KU}DX)$ is finite of order 
\begin{equation*}
\prod_{w\in\mathbb{Z}}\left( \denom\left(\dot{\zeta}^{(w)}_{KU}(k+1,X)\right)\cdot \prod_{\ell\mid n_w}\denom\left( L^{(w,\ell)}_{\tors KU[P^{-1}]}(k+1,X)\right)\right)\end{equation*}
up to powers of primes in $P$.
\end{itemize}
\end{theorem}
\begin{proof}
Throughout, we continue to use the notation $\tilde{\chi}$ from \cref{Defining the... torsion...}: if $\chi$ is a Dirichlet character, then $\tilde{\chi}$ is its associated primitive character.

Fix an integer $w$. Since $n_w = \left|\tors^{(w)} KU^0(X)[P^{-1}]\right|$ is assumed square-free, it is a product of primes. Let $\ell$ be a prime factor of $n_w$, and let $D_{\ell}$ denote the set of Dirichlet characters $\chi\in \Dir_{prime}(n_w^2)$ of conductor $\ell^2$. Then $\{ \tilde{\chi}: \chi\in D_{\ell}\}$ is precisely the set $\Dir_{\mathbb{Q}(\zeta_\ell)}(\ell^2)[\ell]^*$ of elements of order exactly $\ell$ in the group $\Dir_{\mathbb{Q}(\zeta_\ell)}(\ell^2)$ of $\mathbb{Q}(\zeta_\ell)$-valued Dirichlet characters of modulus $\ell^2$. 
Using Carlitz's estimates \cite{MR0109132},\cite{MR0104630} on $p$-adic valuations of generalized Bernoulli numbers, it was shown in \cite{rmjpaper} that, for any positive integer $k$, the product $\prod_{\tilde{\chi} \in \Dir_{\mathbb{Q}(\zeta_\ell)}(\ell^2)[\ell]^*} L(1-k,\tilde{\chi})$ is a rational number whose denominator is given by:
\begin{align}
\label{rmj denom} \denom\left(\prod_{\tilde{\chi} \in \Dir_{\mathbb{Q}(\zeta_\ell)}(\ell^2)[\ell]^*} L(1-k,\tilde{\chi})\right)
  &= \left\{ \begin{array}{ll} 
   1 &\mbox{if\ } \ell-1\nmid k \\
   \ell &\mbox{if\ } \ell-1\mid k.\end{array}\right.
\end{align}

Consequently, for positive integers $k$, the denominator of $L_{\tors KU^0(X)[P^{-1}]}^{(w)}(1-k,X)$ is equal to the product, over all prime factors $\ell$ of $n_w$, of the numbers
\begin{equation*}
\left\{ \begin{array}{ll} 
   1 &\mbox{if\ } \ell-1\nmid k+w\\
   \ell &\mbox{if\ } \ell-1\mid k+w.\end{array}\right.
\end{equation*}

We must compare that denominator to the order of $\pi_{2k-1}(L_{KU}X)$.
To do this, we will work one prime at a time. Fix a prime factor $\ell$ of the order of $\tors KU^0(X)[P^{-1}]$. Consider the homotopy fixed-point/descent spectral sequence (this is the $n=1$ case of the constructions in \cite{MR2030586}; see that paper and its discussion of the relationship to the earlier constructions of \cite{MR0315016}):
\begin{align}
\label{hfpss 1} E_2^{s,t} \cong H^s_c\left( \hat{\mathbb{Z}}_{\ell}^{\times}; \hat{KU}_{\ell}^{-t}(X)\right) &\Rightarrow \pi_{t-s}\left(F(X,\hat{KU}_{\ell})^{\hat{\mathbb{Z}}_{\ell}^{\times}}\right) \\
\nonumber &\cong \pi_{t-s}F\left(X,L_{K(1)}S\right) \\
\nonumber d_r: E_r^{s,t} &\rightarrow E_r^{s+r,t+r-1},
\end{align}
where $K(1)$ is the $\ell$-primary height $1$ Morava $K$-theory spectrum. To be clear, the notation $F(X,Y)$ denotes the mapping spectrum of maps from $X$ to $Y$, so that $\pi_nF(X,Y) \cong [\Sigma^n X,Y] \cong Y^{-n}(X)$.

We need to make an analysis of what the $E_2$-term of spectral sequence \eqref{hfpss 1} looks like, under the stated hypotheses on $X$. For each integer $k$, we have the short exact sequence of graded $\hat{\mathbb{Z}}_{\ell}[\hat{\mathbb{Z}}_{\ell}^{\times}]$-modules
\begin{equation}\label{ses 03w49} 0 
 \rightarrow \left(\tors KU^{-2k}(X)[P^{-1}]\right)^{\wedge}_{\ell}
 \rightarrow \hat{KU}_{\ell}^{-2k}(X)[P^{-1}] 
 \rightarrow \left(\torsfree KU^{-2k}(X)[P^{-1}]\right)^{\wedge}_{\ell}
 \rightarrow 0.\end{equation}

Consider the long exact sequence induced in $H^s_c(\hat{\mathbb{Z}}_{\ell}^{\times}; -)$ by the short exact sequence \eqref{ses 03w49}. The group $H^s_c(\hat{\mathbb{Z}}_{\ell}^{\times}; \tors\hat{KU}^{-2k}_{\ell}(X))$ is trivial if $s>1$, since it is the continuous cohomology of a profinite group $\hat{\mathbb{Z}}_{\ell}^{\times}\cong \mathbb{F}_{\ell}^{\times}\times \hat{\mathbb{Z}}_{\ell}$ of $\ell$-cohomological dimension $1$ (see for example Corollary 2 of section I.4 of \cite{MR1867431} for this standard argument).
We claim that $H^0_c\left(\hat{\mathbb{Z}}_{\ell}^{\times}; \torsfree \hat{KU}^{-2k}_{\ell}(X)\right)$ is also trivial if $k>-a$ or $k<-b$. The argument is simply that the action of the Adams operations on $\torsfree \hat{KU}^{-2k}_{\ell}(X)$ is detected on the rationalization 
\begin{align*} \mathbb{Q}\otimes_{\mathbb{Z}} \torsfree \hat{KU}^{-2k}_{\ell}(X) 
  &\cong \mathbb{Q}\otimes_{\mathbb{Z}} \hat{KU}^{-2k}_{\ell}(X),
\end{align*}
and the fact that the Adams operations act diagonally on rational $K$-theory (the same argument, essentially by the Chern character, as we used in the proof of Theorem \ref{thm on existence of zeta functions}) implies vanishing unless $-b\leq k\leq -a$.

Consequently \eqref{ses 03w49} yields an isomorphism 
\begin{align}
 H^0_c(\hat{\mathbb{Z}}_{\ell}^{\times}; \tors \hat{KU}^{-2k}_{\ell}(X))
  &\cong H^0_c(\hat{\mathbb{Z}}_{\ell}^{\times}; \hat{KU}^{-2k}_{\ell}(X)) 
\end{align}
and a short exact sequence
\begin{equation}
 0 \rightarrow
  H^1_c(\hat{\mathbb{Z}}_{\ell}^{\times}; \tors \hat{KU}^{-2k}_{\ell}(X)) \rightarrow
  H^1_c(\hat{\mathbb{Z}}_{\ell}^{\times}; \hat{KU}^{-2k}_{\ell}(X)) \rightarrow
  H^1_c(\hat{\mathbb{Z}}_{\ell}^{\times}; \torsfree \hat{KU}^{-2k}_{\ell}(X)) \rightarrow 
  0
\end{equation}
for each $k>-a$ and each $k<-b$. 
 
Another consequence of the vanishing of $H^s_c(\hat{\mathbb{Z}}_{\ell}^{\times}; -)$ for $s\geq 2$ is that the spectral sequence \eqref{hfpss 1} can have no nonzero differentials. Hence we have equalities\footnote{Readers with familiarity with Iwasawa theory will likely notice that this proof, particularly \eqref{h0 eq} and \eqref{h1 eq}, bears a close resemblance to the kinds of manipulation of Iwasawa modules and their cohomology found in, for example, \cite{MR0692110}. We suggest that future work of the kind appearing in this paper---i.e., establishing relationships between orders of Bousfield-localized stable homotopy groups, and special values of $L$-functions---would likely benefit from using the techniques of Iwasawa theory as a natural intermediary between stable homotopy groups (by expressing the input for descent spectral sequences along the lines of \eqref{hfpss 1} in terms of cohomology of Iwasawa modules) and special values of $L$-functions.}
\begin{align}
\label{h0 eq} \left|\pi_{2k}F\left(X,L_{K(1)}S\right)\right|
  &= \left| H^0_c\left(\hat{\mathbb{Z}}_{\ell}^{\times}; \tors \hat{KU}^{-2k}_{\ell}(X)\right)\right| \mbox{\ \ \ and} \\
\label{h1 eq} \left|\pi_{2k-1}F\left(X,L_{K(1)}S\right)\right|
  &= \left| H^1_c\left(\hat{\mathbb{Z}}_{\ell}^{\times}; \tors \hat{KU}^{-2k}_{\ell}(X)\right)\right|\cdot\left| H^1_c\left(\hat{\mathbb{Z}}_{\ell}^{\times}; \torsfree \hat{KU}^{-2k}_{\ell}(X)\right)\right|
\end{align}
for each $k>-a$ and each $k<-b$. 

Assuming that $k>-a+1$, we claim that the $\ell$-adic valuation of $\left| H^1_c\left(\hat{\mathbb{Z}}_{\ell}^{\times}; \torsfree \hat{KU}^{-2k}_{\ell}(X)\right)\right|$ agrees with the $\ell$-adic valuation of the product
\begin{equation}\label{zeta value product} \prod_{w\in\mathbb{Z}} \denom\left(\dot{\zeta}_{KU}^{(w)}(1-k,X)\right)\end{equation}
since both have the same order as the stable homotopy group $\pi_{2k-1}F(X[Q^{-1}],L_{K(1)}S)$, where $Q$ is the set of primes
\[ \{2\} \cup\left\{ p: p\mbox{\ divides\ } \left| \tors KU^0(X)[P^{-1}]\right|\right\}.\]
To see this, observe that the map $X\rightarrow X[Q^{-1}]$ induces an Adams-operation-preserving isomorphism 
\begin{align*}
 \left(\torsfree\ KU^0(X)\right)[Q^{-1}]
  &\stackrel{\cong}{\longrightarrow}
 KU^0(X)[Q^{-1}],
\end{align*}
and the homotopy fixed-point spectral sequence
\begin{align}
\label{hfpss 2} E_2^{s,t} \cong H^s_c\left( \hat{\mathbb{Z}}_{\ell}^{\times}; \hat{KU}_{\ell}^{-t}(X[Q^{-1}])\right) &\Rightarrow \pi_{t-s}F\left(X[Q^{-1}],L_{K(1)}S\right)
\end{align}
collapses at the $E_2$-page with no nonzero differentials, yielding the isomorphism 
\begin{align*}
 H^1_c\left(\hat{\mathbb{Z}}_{\ell}^{\times}; \torsfree \hat{KU}^{-2k}_{\ell}(X)\right) 
  &\cong \pi_{2k-1}F\left(X[Q^{-1}],L_{K(1)}S\right).
\end{align*}

The following type of argument is standard in stable homotopy theory, but we have chosen to present it in more detail than that audience would find necessary, since the author hopes that some number theorists may read this paper too. Write $E(n)$ for the $\ell$-primary height $n$ Johnson-Wilson theory. The ``fracture square''\footnote{The homotopy pullback square \eqref{fracture sq 1} is classical. See Bauer's chapter \cite{bauer2011bousfield} in the book \cite{douglas2014topological} for a nice write-up.}
\begin{equation}\label{fracture sq 1}\xymatrix{
 L_{E(1)}S \ar[r] \ar[d] & L_{K(1)}S \ar[d] \\
 L_{E(0)}S \ar[r] & L_{E(0)}L_{K(1)}S 
}\end{equation}
yields a homotopy fiber sequence
\begin{equation}\label{fib seq 094}  L_{E(1)}S \rightarrow L_{K(1)}S \vee L_{E(0)}S \rightarrow L_{E(0)}L_{K(1)}S .\end{equation}
Bousfield localization at $E(0)$ is simply rationalization, so $L_{E(0)}S$ is simply the Eilenberg-Mac Lane spectrum $H\mathbb{Q}$ representing rational cohomology. Running the homotopy fixed-point spectral sequence \eqref{hfpss 1} for the sphere to calculate $\pi_*(L_{K(1)}S)$ yields that $L_{E(0)}L_{K(1)}S$ splits as a wedge $H\mathbb{Q}_p\vee \Sigma^{-1}H\mathbb{Q}_p$, i.e., mapping into $L_{E(0)}L_{K(1)}S$ yields two copies of $p$-adic rational cohomology, with one copy shifted in degree by $1$.
Consequently the homotopy fiber of the map $L_{E(1)}S \rightarrow L_{K(1)}S$ is weakly equivalent to $\Sigma^{-1}H(\mathbb{Q}_p/\mathbb{Q}) \vee \Sigma^{-2} H\mathbb{Q}_p$.
Consequently, the long exact sequence obtained by applying $[X[Q^{-1}],-]$ to \eqref{fib seq 094} degenerates to an isomorphism 
\begin{align*}
 \pi_{2k-1}F\left(X[Q^{-1}],L_{K(1)}S\right)
  &\cong\pi_{2k-1}F\left(X[Q^{-1}],L_{E(1)}S\right)
\end{align*}
for $k>-a$ and for $k<-b$, since the rational cohomology of $X$ vanishes in the relevant degrees. 
Since the $\ell$-localization of $L_{KU}S$ is $L_{E(1)}S$, the $\ell$-adic valuation of $\left| \pi_{2k-1}F(X[Q^{-1}],L_{KU}S)\right|$ is equal to the $\ell$-adic valuation of $\left| \pi_{2k-1}F(X[Q^{-1}],L_{E(1)}S)\right|$
 The localization of $\pi_{2k-1}F(X[Q^{-1}],L_{KU}S)$ at $\ell$ is isomorphic to $\pi_{2k-1}F(X[Q^{-1}],L_{E(1)}S)$. Finally, if $k>-a+1$, then Theorem \ref{main spec vals thm} gives the equality of the order of $\pi_{2k-1}F(X[Q^{-1}],L_{KU}S)\cong \pi_{2k-1}L_{KU}DX[Q^{-1}]$ with \eqref{zeta value product}.
Hence the factor $\left| H^1_c\left(\hat{\mathbb{Z}}_{\ell}^{\times}; \torsfree \hat{KU}^{-2k}_{\ell}(X)\right)\right|$ in \eqref{h1 eq} is accounted for by the $\ell$-adic valuation of the denominator of $\dot{\zeta}_{KU}(1-k,X)$, when $k>1-a$.

We still need to relate the factors $\left| H^s_c\left(\hat{\mathbb{Z}}_{\ell}^{\times}; \tors \hat{KU}^{-2k}_{\ell}(X)\right)\right|$ in \eqref{h0 eq} and \eqref{h1 eq} to $L_{\tors KU[P^{-1}]}(1-k,X)$. For this we use the finite filtration 
\begin{equation}\label{l-adic skel filt 1} \left(\tors KU^{-2k}(X)[P^{-1}]\right)^{\wedge}_{\ell} = F_k^a \supseteq F_k^{a+1} \supseteq F_k^{a+2} \supseteq \dots\supseteq F_k^{b-1} \supseteq F_k^b = 0\end{equation}
of $\tors KU^{-2k}(X)[P^{-1}]$, where $F_k^t$ is the $\ell$-adic completion of \[ \ker\left( \tors KU^{-2k}(X)[P^{-1}]\rightarrow \tors KU^{-2k}(X^{2t})[P^{-1}] \right).\] That is, \eqref{l-adic skel filt 1} is the $\ell$-adic completion of the skeletal filtration \eqref{skel filt 1} on $KU^{-2k}$.
Applying the continuous group cohomology functor $H^*_c(\hat{\mathbb{Z}}_{\ell}^\times; -)$ to \eqref{l-adic skel filt 1} yields the strongly convergent spectral sequence
\begin{align}
\label{l-adic skel filt ss 1} E_1^{s,t} \cong H^s_c(\hat{\mathbb{Z}}_{\ell}^\times; F_k^t/F_k^{t+1})
  &\Rightarrow H^s_c(\hat{\mathbb{Z}}_{\ell}^\times; \tors \hat{KU}_{\ell}^{-2k}(X)) \\
\nonumber d_r: E_r^{s,t} &\rightarrow E_r^{s+1,t+r}.
\end{align}
The order of the abelian group $F_k^t/F_k^{t+1}$ is equal to the largest power of $\ell$ which divides $n_t$. Since we have assumed that $n_t$ is square-free for all $t$, the abelian group $F_k^t/F_k^{t+1}$ must be either trivial or a one-dimensional $\mathbb{F}_{\ell}$-vector space. In the latter case, the pro-$\ell$-Sylow subgroup of $\hat{\mathbb{Z}}_{\ell}^\times$ must act trivially, and consequently the action of $\hat{\mathbb{Z}}_{\ell}^\times$ on $\mathbb{F}_{\ell}$ is determined by the action of the quotient $\mathbb{F}_{\ell}^{\times}$ of $\hat{\mathbb{Z}}_{\ell}^\times$. There are $p-1$ possible actions of $\mathbb{F}_{\ell}^{\times}$ on $\mathbb{F}_{\ell}$. Of those $p-1$ actions, there are $p-2$ which fix only the zero element, and consequently have trivial group cohomology in all degrees. The single cohomologically nontrivial action of $\mathbb{F}_{\ell}^{\times}$ on $\mathbb{F}_{\ell}$ is the trivial action. We will write $\mathbb{F}_{\ell}^{\triv}$ for $\mathbb{F}_{\ell}$ with the resulting trivial $\hat{\mathbb{Z}}_{\ell}^\times$-action. Then $H^*_c(\hat{\mathbb{Z}}_{\ell}^\times;\mathbb{F}_{\ell}^{\triv})\cong \Lambda(h)$, an exterior $\mathbb{F}_{\ell}$-algebra on a single generator $h$ in degree $1$. Consequently the $E_1$-page of \eqref{l-adic skel filt ss 1} is described as follows:
\begin{itemize}
\item $E_1^{s,t}$ is trivial if $s\neq 0,1$,
\item $E_1^{0,t}\cong E_1^{1,t}$ for all $t$,
\item $E_1^{0,t}$ is a one-dimensional $\mathbb{F}_{\ell}$-vector space if the action of the Adams operations on $\tors^{(t)}\ \hat{KU}_{\ell}^{-2k}(X)$ is trivial,
\item and $E_1^{0,t}$ is trivial otherwise.
\end{itemize}

Consider the direct sum of the spectral sequences \eqref{l-adic skel filt ss 1} over all integers $k$. Call this the ``torsion spectral sequence.'' The cohomology $H^*_c(\hat{\mathbb{Z}}_{\ell}^\times;\tors\hat{KU}_{\ell}^*(X))$ is periodic under the action of $v_1^{-1}$, i.e., the $(\ell-1)$th power of the Bott class in $KU^2$. Consequently, for each weight $w$ such that $F_k^w/F_k^{w+1}$ is nontrivial, we get a contribution to the $E_1$-page of the torsion spectral sequence 
which is isomorphic to the $E_2$-page of the homotopy fixed-point spectral sequence \begin{align*} E_2^{s,t} \cong H^s_c\left( \hat{\mathbb{Z}}_{\ell}^{\times}; \hat{KU}_{\ell}^{-t}(S^{2w-1}/\ell)\right) &\Rightarrow \pi_{t-s}F\left(S^{2w-1}/\ell,L_{K(1)}S\right)\end{align*}
of the $(2w-1)$-dimensional mod $\ell$ Moore spectrum $S^{2w-1}/\ell$. This is because $\hat{KU}^{-2t}_{\ell}(S^{2w-1}/\ell)\cong \mathbb{F}_{\ell}$ with $\Psi^p$ acting as multiplication by $p^{w + t}$, which is the same as the action of $\Psi^p$ on $F^w_t/F^{w+1}_t$.

Since the mod $\ell$ Moore spectrum is rationally acyclic, the square \eqref{fracture sq 1} gives us the weak equivalences 
\begin{align*} 
 F\left(S^{2w-1}/\ell,L_{K(1)}S\right)
  &\simeq F\left(S^{2w-1}/\ell,L_{KU}S\right) \\
  &\simeq L_{KU}D\left( S^{2w-1}/\ell\right)\\
  &\simeq L_{KU}\left(S^{-2w}/\ell\right).
\end{align*}
Now Proposition \ref{definedness of tors zeta}, together with the results of \cite{rmjpaper} summarized above in \eqref{rmj denom}, yields that the group $\pi_{2k-1}L_{KU}\left(S^{-2w}/\ell\right)$, a factor of the $E_1$-page of the torsion spectral sequence, has order equal to the denominator of $L^{(w,\ell)}_{\tors KU[P^{-1}]}(1-k,X)$. 

The conclusion here is that the orders of the groups in the $E_1$-page of the torsion spectral sequence, in the range of degrees described above, agree with the products of the denominators of special values of isoweight factors of $L_{\tors KU[P^{-1}]}(s,X)$ described in \eqref{product 18a}. Meanwhile, the abutment of the torsion spectral sequence is the summand coming from $K$-theoretic torsion in the input for the descent spectral sequence \eqref{hfpss 1} which converges to $\pi_*(L_{KU}DX)[P^{-1}]$ and has no nonzero differentials.
Consequently \eqref{product 18a} is equal to the order of $\pi_{2k-1}(L_{KU}DX)$ for all $k>1-a$, up to powers of primes in $P$, if and only if the torsion spectral sequence collapses at the $E_1$-page with no nonzero differentials involving a bidegree which contributes 
to $E_2^{p,q}$ in \eqref{hfpss 1} in a degree which contributes to $\pi_{n}(L_{KU}DX)$ in the abutment, with $2n > -a$.

If the skeletal filtration of $\tors KU^0(X)[P^{-1}]$ splits completely, then spectral sequence \eqref{l-adic skel filt ss 1} collapses at the $E_1$-page with no nonzero differentials, so formula \eqref{product 18a} is true, as claimed.
For the converse: suppose that the skeletal filtration of $\tors KU^0(X)[P^{-1}]$ does not split completely. Then for at least one value of $t$, the extension of graded $\hat{\mathbb{Z}}_{\ell}[\hat{\mathbb{Z}}_{\ell}^{\times}]$-modules
\[ 0 \rightarrow \coprod_{k\in\mathbb{Z}} F_k^{t+1} \rightarrow \coprod_{k\in\mathbb{Z}}F_k^t \rightarrow \coprod_{k\in\mathbb{Z}} F_k^t/F_k^{t+1} \rightarrow 0\]
is not split. Since the skeletal filtration on $\tors KU^0(X)[P^{-1}]$ was not assumed to split additively, it is not necessarily true that each $F_k^t$ is an $\mathbb{F}_{\ell}$-vector space. However, by finiteness of $X$, it is at least true that there is an integer $N$ such that $F_k^t$ is an $\mathbb{Z}/\ell^N\mathbb{Z}$-module for all $k$ and all $t$. Consequently the action of the Bott element on $\coprod_{k\in\mathbb{Z}}F_k^t$ is periodic\footnote{That is, the action of some power of the Bott element on $\coprod_{k\in\mathbb{Z}}F_k^t$ is not just an isomorphism of abelian groups---which, after all, is true of multiplication by the Bott element itself---but also an isomorphism of $\hat{\mathbb{Z}}_{\ell}[\hat{\mathbb{Z}}_{\ell}^{\times}]$-modules.} of some period (e.g. period $\ell^{N-1}(\ell-1)$). Consequently, if the sequence of $\hat{\mathbb{Z}}_{\ell}[\hat{\mathbb{Z}}_{\ell}^{\times}]$-modules
\begin{equation}\label{ses fkj04} 0 \rightarrow F_k^{t+1} \rightarrow F_k^t \rightarrow F_k^t/F_k^{t+1}\rightarrow 0 \end{equation}
is nonsplit for some value of $k$, then it is also nonsplit for arbitrarily much higher values of $k$, as well as arbitrarily much lower values of $k$.

By analysis of the torsion spectral sequence, if \eqref{ses fkj04} is nonsplit, then the identity element 
\[ \id\in \Ext^0_{cont. \hat{\mathbb{Z}}_{\ell}[\hat{\mathbb{Z}}_{\ell}^{\times}]-mod}\left( \mathbb{F}_{\ell}^{\triv},  \mathbb{F}_{\ell}^{\triv} \right) \cong E_1^{0,t}\]
supports a nonzero differential of some length. Choose some value of $k$ such that \eqref{ses fkj04} is nonsplit and such that the resulting nonzero differential hits a class which contributes to bidegree $E_2^{p,q}$ in \eqref{hfpss 1} in a degree which contributes to $\pi_{n}(L_{KU}DX)$ in the abutment, with $n>>-a$. That differential causes formula \eqref{product 18a} to fail to agree with the order of $\pi_{2k-1}(L_{KU}DX)$ after localization at $\ell$.

Consequently conditions $1$ and $2$ are equivalent. An entirely analogous argument proves the equivalence of conditions $1$ and $3$.
\end{proof}
One family of special cases of Theorem \ref{main spec vals thm 2} was studied in the paper \cite{rmjpaper}. Using the notation introduced in the present paper\footnote{The notation $L_{KU}(s, S/p)$ from \cite{rmjpaper} corresponds to the notation $\zeta_{KU}(s,\Sigma^{-1} S/p)$ in the present paper. The reason for the desuspension is that the approach in \cite{rmjpaper} was based around formulating $KU$-local $L$-functions whose special values recover orders of $KU$-local stable homotopy groups, while in the more general and natural framework of the present paper, the special values of $KU$-local zeta-functions instead recover the $KU$-local stable {\em cohomotopy} groups of a finite CW-complex $X$, or what comes to the same thing, the $KU$-local stable homotopy groups of the Spanier-Whitehead dual $DX$ of $X$. We have $D(\Sigma^{-1} S/p) \simeq S/p$.}, \cite{rmjpaper} showed that $\zeta_{KU}(s,\Sigma^{-1} S/p) = \zeta_F(s)/\zeta(s)$ for any odd prime $p$, where $S/p$ is the mod $p$ Moore spectrum, and where $F$ is the largest totally real subfield of the cyclotomic field $\mathbb{Q}(\zeta_{p^2})$. 

Recall that Theorem \ref{main spec vals thm} made some mention of {\em regular primes,} i.e., those primes $p$ which do not divide the numerator of $\zeta(1-k)$ for any positive integer $k$. Corollary \ref{main spec vals thm 2 cor}, below, requires the following generalization: given a number field $F$, we say that a prime number $p$ is {\em $F$-irregular} if $p$ divides the numerator of $\zeta_F(1-k)$ for some positive integer $k$. Here $\zeta_F(s)$ is the Dedekind zeta-function of $F$. The $\mathbb{Q}$-irregular primes are simply the classical irregular primes.
\begin{corollary}\label{main spec vals thm 2 cor}
Let $X,P,a,b$ be as in Theorem \ref{main spec vals thm 2}. Suppose that the the skeletal filtration of $\tors KU^0(X)[P^{-1}]$ splits completely. Write $N$ for the order of the group $\tors KU^0(X)[P^{-1}]$.
Then, for all odd integers $2k-1\geq 1-2a$, the order of $\pi_{2k-1}(L_{KU}DX)$ is equal to the denominator of $\zeta_{KU[P^{-1}]}(1-k,X)$, up to powers of primes in $P$ and powers of $F$-irregular primes, where $F$ ranges across all the wildly ramified subfields of the cyclotomic field $\mathbb{Q}\left(\zeta_{N^2}\right)$.
\end{corollary}
\begin{proof}
The Dedekind zeta-function $\zeta_{\mathbb{Q}\left(\zeta_{N^2}\right)}(s)$ factors as the product of the Dirichlet $L$-functions of the primitive Dirichlet characters $\tilde{\chi}$, where $\chi$ ranges across the Dirichlet characters of modulus $N^2$. For a prime divisor $\ell$ of $N$, the Dirichlet characters $\chi$ on $\mathbb{Z}/\ell^2\mathbb{Z}^{\times}$ which vanish on the complementary summand of $\Syl(\ell)$ are those such that the $L$-function of $\tilde{\chi}$ is a factor of the Dedekind zeta-function of a subfield of $\mathbb{Q}\left(\zeta_{N^2}\right)$ in which $\ell$ ramifies wildly. This material is classical; e.g. see Corollary 3.6 and Theorem 4.3 from \cite{MR1421575}. As a consequence, the same argument (about cancellation of factors in numerators with factors in denominators) used in the proof of Theorem \ref{main spec vals thm} suffices here to establish the second claim.
\end{proof}

\begin{remark}
In \eqref{product 18a}, it was important that we take the product, over $w$ and $\ell$, of the denominators of the special values of the weight $w$ conductor $\ell$ $KU$-local zeta-factor. We do {\em not} get the same result if we simply take the denominator of $\zeta_{KU[P^{-1}]}(1-k,X)$. In Remark \ref{cancellation remark 1}, we already saw that a prime factor of the denominator of the weight $w_1$ factor of a $KU$-local provisional zeta-function can cancel with a prime factor of the {\em numerator} of the weight $w_2$ factor, for $w_1\neq w_2$. The same phenomenon occurs with the (non-provisional) $KU$-local zeta-function.

Even within a single weight $w$, it is possible for a prime factor of the denominator of the weight $w$ conductor $\ell_1$ factor of a $KU$-local zeta-function to cancel with a prime factor of the numerator of the weight $w$ conductor $\ell_2$ factor, for $\ell_1\neq \ell_2$. An example occurs already for the desuspended mod $21$ Moore spectrum $\Sigma^{-1} S/21$. We have $KU^0(\Sigma^{-1} S/21) = \tors KU^0(\Sigma^{-1} S/21) \cong \mathbb{Z}/21\mathbb{Z}$, all in weight zero. 
Let $R_3$ (respectively, $R_7$) denote the set of prime representations of $\tors KU^0(\Sigma^{-1}S/21)$ induced up from the order $3$ subgroup (respectively, the order $7$ subgroup). Similarly, write $D_3$ (respectively, $D_7$) for the set of Dirichlet characters of conductor $3$ (respectively, conductor $7$) and modulus $21$ which vanish on the complementary summand of $\Syl(21)$.
Unwinding the equalities from Theorem \ref{main spec vals thm 2}, we have\footnote{The calculations \eqref{eq f34oim4} appeared already in a slightly different context in \cite{rmjpaper}. Computer calculation of these special values of Dirichlet $L$-functions, via generalized Bernoulli numbers (see for example section 1.2 of \cite{MR0360526}), is straightforward in SageMath \cite{sagemath66} or in Magma \cite{MR1484478}.}
\begin{align}
\nonumber \zeta_{KU}(-5,\Sigma^{-1}S/21) 
  &= \prod_{p} \prod_{\rho\in R_3\cup R_7}\frac{1}{\det\left(\id - p^{w-s}\Psi_{\rho_w,p,i_w}(p)\right)} \\
\nonumber  &= L_{\tors KU}^{(0,3)}(-5,\Sigma^{-1}S/21)\cdot L_{\tors KU}^{(0,7)}(-5,\Sigma^{-1}S/21) \\
\nonumber  &= \left(\prod_{\chi\in D_3}L(-5,\tilde{\chi})\right)\cdot \left(\prod_{\chi\in D_7}L(-5,\tilde{\chi})\right) \\
\label{eq f34oim4}  &= \frac{2^2\cdot 7 \cdot 43 \cdot 1171}{3} \cdot \frac{2^6\cdot 138054547\cdot 163933047708171216095114393777711}{7} \\
 \nonumber &= \frac{2^8\cdot 43 \cdot 1171\cdot 138054547\cdot 163933047708171216095114393777711}{3},
\end{align}
whose denominator is $3$. Meanwhile, we have
\begin{align*}
 \denom\left( L_{\tors KU}^{(0,3)}(-5,\Sigma^{-1}S/21)\right) & \\
  \ \ \ \ \ \cdot \denom\left(L_{\tors KU}^{(0,7)}(-5,\Sigma^{-1}S/21)\right) 
  &= \denom\left(\frac{2^2\cdot 7 \cdot 43 \cdot 1171}{3}\right) \\
  &\ \ \ \cdot \denom\left(\frac{2^6\cdot 138054547\cdot 163933047708171216095114393777711}{7}\right)\\
  &= 21\\
  &= \left| \pi_{-13}(L_{KU}S/21))\right|.
\end{align*}
The trouble is that, while $7$ is a regular prime, it is also $F$-irregular, where $F$ is the minimal subfield of $\mathbb{Q}(\zeta_9)$ in which $3$ ramifies wildly. 

This example demonstrates that Corollary \ref{main spec vals thm 2 cor} is perhaps not of great practical use, except in very special cases: the trouble is that there are simply many more $F$-irregular primes than classical irregular primes. Indeed, the prime $2$ is already $F$-irregular, for the same number field $F$ described in the previous paragraph.

The example $X = \Sigma^{-1}S/21$ is minimal in the sense that it minimizes the prime factors ($3$ and $7$) of the order of the torsion in $KU^0(X)$. Nevertheless, the relatively large prime factors $138054547$ and $163933047708171216095114393777711$ occured in the numerator of \eqref{eq f34oim4}. Similarly large (or, in fact, much larger) prime factors are often found in numerators of special values of $\zeta_{KU}(s,X)$ for CW-complexes $X$ that have nontrivial torsion in $K$-theory. 
\end{remark}

\bibliography{/home/asalch/texmf/tex/salch}{}
\bibliographystyle{plain}
\end{document}